\documentclass[12pt,twoside]{amsart} \usepackage{amssymb,color,tikz}
\usetikzlibrary{matrix} \nonstopmode \textwidth=16.00cm

\usepackage{xcolor}
\usepackage{amsmath, graphicx}

\numberwithin{equation}{section} 

 \newcommand {\PP}{\mathbb{P}}
\DeclareMathOperator{\Soc}{soc} \def\cocoa{{\hbox{\rm C\kern-.13em
      o\kern-.07em C\kern-.13em o\kern-.15em A}}}
\newtheorem{theorem}{Theorem}[section]
\newtheorem{lemma}[theorem]{Lemma}
\newtheorem{proposition}[theorem]{Proposition}
\newtheorem{corollary}[theorem]{Corollary}
\newtheorem{conjecture}[theorem]{Conjecture} \theoremstyle{definition}
\newtheorem{definition}[theorem]{Definition}
\newtheorem{remark}[theorem]{Remark}
\newtheorem{example}[theorem]{Example}
\newtheorem{notation}[theorem]{Notation}
\definecolor{MyDarkGreen}{cmyk}{0.7,0,1,0}

\begin{document}
\title[On the Weak Lefschetz Property for height four equigenerated complete intersections ]%
{On the Weak Lefschetz Property for height four equigenerated complete intersections}
\author[M.\ Boij]{Mats Boij} \address{Department of Mathematics, KTH
  Royal Institute of Technology, S-100 44 Stockholm, Sweden}
\email{boij@kth.se}
\author[J.\ Migliore]{Juan Migliore${}^*$}
\address{ Department of Mathematics, University of Notre Dame, Notre
  Dame, IN 46556, USA} \email{migliore.1@nd.edu}

  \author[R.\
M.\ Mir\'o-Roig]{Rosa M.\ Mir\'o-Roig${}^{**}$} \address{Facultat de
  Matem\`atiques, Department d'\`Algebra i Geometria, Gran Via des les
  Corts Catalanes 585, 08007 Barcelona, Spain} \email{miro@ub.edu}
\author[U.\ Nagel]{Uwe Nagel${}^{***}$} \address{Department of
  Mathematics, University of Kentucky, 715 Patterson Office Tower,
  Lexington, KY 40506-0027, USA} \email{uwe.nagel@uky.edu}

\thanks{
  We thank  KTH, Institut Henri Poincar\'e in Paris and Centro Internazionale per la Ricerca Matematica in Trento for their hospitality and support.
  Boij was partially supported by the grant VR2013-4545;
Migliore was partially supported by grants from the Simons Foundation
  (\#208579 and \#309556);
Mir\'o-Roig was partially supported by the grant PID2019-104844GB-I00; Nagel  was partially supported by grants from the Simons Foundation (\#317096  and \#636513).  \\
\indent We thank the referee for carefully reading the manuscript and making useful comments.
} \pagestyle{plain}
\begin{abstract}
We consider the conjecture that all artinian height 4 complete intersections of forms of the same degree $d$ have the Weak Lefschetz Property (WLP). We translate this problem to one of studying the general hyperplane section of a certain smooth curve in $\mathbb P^3$, and our main tools are the Socle Lemma of Huneke and Ulrich together with a careful liaison argument. Our main results are (i) a proof that the property holds for $d=3,4$ and 5;  (ii) a partial result showing maximal rank in a non-trivial but incomplete range, cutting in half the previous unknown range; and (iii)  a proof that maximal rank holds in a different range, even without assuming that all the generators have the same degree. We furthermore conjecture that if there were to exist any height 4 complete intersection generated by forms of the same degree and failing WLP then there must exist one (not necessarily the same one) failing by exactly one (in a sense that we make precise). Based on this conjecture we outline an approach to proving WLP for all equigenerated complete intersections in four variables.  Finally, we apply our results to the Jacobian ideal of a smooth surface in~$\mathbb P^3$.
\end{abstract}
\maketitle
\section{Introduction}
Let $A = R/I$, where $R = k[x_1,\dots,x_n]$ is the polynomial ring in $n$ variables, $k$ is an algebraically closed field of characteristic zero and $I$ is a homogeneous ideal such that $R/I$ is artinian. We say that $A$ has the {\it Weak Lefschetz Property (WLP)} if, for a general linear form $L$, the homomorphism $\times L : [A]_t \rightarrow [A]_{t+1}$ has maximal rank for all $t$.

When $I = \langle F_1,
\dots, F_n \rangle$ is a complete intersection, we know that for a general choice of homogeneous polynomials $F_1,\dots, F_n$, $R/I$ has the WLP thanks to a result of \cite{stanley}, \cite{watanabe} and
\cite{RRR}.
In \cite{HMNW}, it was shown that when $n=3$, {\it every} complete
intersection $R/(F_1,F_2,F_3)$ has the WLP, and previously it was known for $n=2$.  It has been asked and conjectured by many authors whether every complete intersection has the WLP (and also whether every complete intersection has the related Strong Lefschetz Property, which we do not consider here). The first occurrence that we could find of this conjecture is in \cite{RRR} from 1991.

In four or more variables very little is known beyond the results mentioned above. In this paper we consider the case of four variables. For most of this paper we specialize to the situation where all generators have the same degree, $d$, although our last result is for arbitrary degrees. As an application, we look at the special case where $I$ is the Jacobian ideal of a smooth surface of degree $d+1$ in $\mathbb P^3$.

We begin by translating the problem to a much more geometric setting. If $I = \langle F_1,F_2,F_3,F_4\rangle$, where the generators have degree $d$, we say that $I$ is a {\it complete intersection of type $(d,d,d,d)$}. Choose two {\it general} 2-dimensional subspaces of the 4-dimensional subspace spanned by these generators. These  define two disjoint smooth complete intersection curves $C_1, C_2$ in $\mathbb P^3$, and we show that the Hartshorne-Rao module of $C = C_1 \cup C_2$ is exactly $A$, viewed now as an $R$-module. Most of the paper focuses on the intersection, $Z$, of $C$ with a general hyperplane $H$.  In the coordinate ring $S$ of $H$ let $I_Z$ be the homogeneous ideal of $Z$ and let $B = S/I_Z$. Let $\overline B$ be a general artinian reduction of $B$. We find a Hilbert function for $Z$ that is equivalent to the possession of the WLP for $A$, and indeed the set of possible Hilbert functions of $Z$, if $A$ were to fail to have WLP, is at the heart of our approach.

An important tool that we will use to analyze the Lefschetz behavior of $A$ is the beautiful Socle Lemma of Huneke and Ulrich \cite{HU}. In Section \ref{sect tools} we give the following translation of the Socle Lemma (following the approach of Huneke and Ulrich in their paper): the socle of $\overline B$ starts no later than the least degree of a form vanishing on $Z$ that does not lift to $C$. This is crucial to our work.

Another important tool is the fact that the general hyperplane section of a reduced, irreducible curve (in our case $C_1$ and $C_2$ separately) has the Uniform Position Property (UPP) (\cite{harris2} and \cite{EH}). Even though $Z$ does not have UPP (it is just the union of two sets, $Z = Z_1 \cup Z_2$, of the same size such that each has UPP), we show that the Hilbert function of $Z$ is of decreasing type in Proposition~\ref{decreasing type}.

Our first main result, Theorem \ref{d1=d2=d3=d4}, gives a new range where maximal rank must hold: multiplication by a general linear form is injective from degree $t-1$ to degree $t$ for all $t < \lfloor \frac{3d+1}{2} \rfloor $ (with a corresponding statement for surjectivity, thanks to duality). We note that when $I$ is the Jacobian ideal of a smooth surface in $\mathbb P^3$, Ilardi proved injectivity for $t \leq d$. Alzati and Re subsequently proved the same result without assuming that $I$ is a Jacobian ideal. On the other hand, the full WLP result is equivalent to injectivity for $t=2d-2$. Thus our result essentially cuts in half the range that was open.

Our next results prove that WLP holds for $d = 3$ (Proposition \ref{d=3}), $d=4$ (Theorem \ref{d=4}) and $d=5$ (Corollary \ref{d=5}), all of which are new. It was already shown in  \cite{MN-quadrics} that WLP holds when $d=2$, and beyond $d=5$ the number of possible cases grows too large to handle in a reasonable way. The case $d=3$ follows immediately from Theorem \ref{d1=d2=d3=d4}. The cases $d=4,5$, on the other hand, are proved using a new approach involving a careful series of links applied to the hyperplane section $Z$, which we view simultaneously as a series of links starting with $Z_1$ together with a series of links starting with $Z_2$.


The links that we use are very balanced, treating $Z_1$ and $Z_2$ in exactly the same way, in order to utilize symmetry. In fact, we use two ``parallel" sequences of linked schemes $Z_1 = Y_{0,1}, Y_{1,1},\ldots$ and $Z_2 = Y_{0,2}, Y_{1,2},\ldots$ such that, for each $i$, the fact that $Z_1$ and $Z_2$ are indistinguishable both geometrically and numerically is also true for $Y_{i,1}$ from $Y_{i,2}$. We refer to this idea as  the {\it Symmetry Principle}  (see (\ref{symm princ})), and it is the key to our conclusions involving the links.

In section  \ref{strategy} we outline an approach to prove that every complete intersection of type $(d,d,d,d)$ has the WLP, for $d \geq 6$. It extends the arguments we used to prove the cases $d=4$ and $d=5$. There are two steps missing to prove the  full result, which we highlight as Conjectures \ref{force exactly one} and \ref{get contra}.  The idea is to make very careful calculations involving the minimal free resolutions and the $h$-vectors of all the sets in the series of links and show how it should be possible to reach a contradiction of the Symmetry Principle, thus showing that WLP holds. We have not been able to prove these conjectures for the  general case, but we were able to get around them for the cases $d=4,5$.

In subsection \ref{jacobian ideals}, we apply our results to the case of Jacobian ideals of smooth surfaces in $\mathbb P^3$. In addition to the observations made above, we observe that every smooth hypersurface in $\mathbb P^3$ of degree 3, 4, 5 or 6 has a Jacobian ideal that has the WLP, improving the known range.  Finally, in subsection \ref{not equigenerated subsec}, we use completely different methods to prove an injectivity result for arbitrary complete intersections, removing the assumption that all generators have the same degree.  More precisely, if the generator degrees of $I$ are $d_1,d_2,d_3,d_4$ then we set $d_1 + d_2 + d_3 + d_4 = 3 \lambda + r$, $0 \leq r \leq 2$, and prove that the multiplication by a general linear form is injective for $t< \lambda$.  The restriction to the case $d_1 = \dots = d_4 = d$ is not as strong as the result in Theorem \ref{d1=d2=d3=d4}, but it is still stronger than the Alzati-Re result and in any case it omits the restriction that the ideal is equigenerated.

\section{Some tools for height four complete intersections} \label{sect tools}

Assume from now on that $R = k[x_1,x_2,x_3,x_4]$ is the polynomial ring in four variables where $k$ is an algebraically closed field of characteristic zero, and $I = \langle F_1,F_2,F_3,F_4 \rangle \subset R$ is a homogeneous complete intersection, with $\deg F_i = d$ for $i = 1,2,3,4$ (except for Proposition \ref{d1>=d2>=d3>=d4}).
If $L$ is a general linear form, it defines a hyperplane $H$. Let $S = R/\langle L \rangle \cong k[x,y,z]$ be the coordinate ring of $H$. If $L'$ is another general linear form and $Z$ is a zero-dimensional subscheme of $H$ then we set $T = S/\langle L' \rangle \cong k[x,y]$ and we recall that an artinian reduction of $S/I_Z$ has the same graded Betti numbers over $T$ as $S/I_Z$ has over $S$.

Let $I_1 = \langle F_1,F_2 \rangle$ and $I_2 = \langle F_3,F_4 \rangle$ and let $C$ be the curve defined by the  ideal $I_C =
I_1 \cap I_2$.  We make the following observations about $C$.

\begin{lemma} \label{basic facts about C}
\begin{itemize}
\item[(a)] $I_C$ is saturated and $C = C_1 \cup C_2$ in
  $\mathbb P^3 = Proj(R)$, where $C_1$ and $C_2$ are the disjoint
  complete intersections defined by $I_1$ and $I_2$ respectively.

\item[(b)] $I_C = I_1 \cdot I_2$.

\item[(c)] The minimal free resolution of $I_C$ is obtained as the tensor
  product of the Koszul resolutions of $I_1$ and $I_2$.  Hence in
  particular, the Hilbert function of $C$ is completely determined.

\item[(d)] For the Hartshorne-Rao module $M(C) = \bigoplus_{t \in \mathbb Z} H^1(\mathcal I_C (t))$ we have
  \[
  M(C) \cong A = R/\langle F_1,F_2,F_3,F_4 \rangle.
  \]
In particular, the minimal free resolution of $M(C)$ is given by the
  Koszul resolution.

\end{itemize}
\end{lemma}

\begin{proof}
It is clear that $I_C$ is saturated. The curves are disjoint since $I$ is artinian. This proves (a). Part (b) follows from \cite{serre} Corollaire, page 143, since $C_1$ and $C_2$ are ACM curves in  $\mathbb P^3$.  Part (c) is \cite{MDP} Corollaire 7.6. For (d), from the exact sequence
  \[
  0 \rightarrow I_C \rightarrow I_1 \oplus I_2 \rightarrow I_1 + I_2
  \rightarrow 0,
  \]
  sheafifying and taking cohomology, we obtain the Hartshorne-Rao
  module as claimed. Note that the sheafification of $I_1 + I_2$ is $\mathcal O_{\mathbb P^3}$ since $C_1$ and $C_2$ are disjoint, and that $I_1 + I_2 = I$.
\end{proof}

\begin{remark} \label{bertini} By successive use of Bertini's theorem (see for instance \cite{kleiman}) we can assume
    that all the $F_i$ are smooth, and that both $C_1$ and $C_2$ are
    smooth and irreducible.
\end{remark}

Let $L$ be a general linear form and let $H$ be the hyperplane
defined by $L$.  Let $Z$ be the zero-dimensional scheme cut out on
$C$ by $H$.  As a subscheme in $\mathbb P^3$, $Z$ has a homogeneous
ideal that we will denote $I_Z$, and as a subscheme of $H$ it has a
homogeneous ideal $I_{Z|H}$.  Consider the exact sequence of sheaves
\[
0 \rightarrow \mathcal I_C(t-1) \stackrel{\times L}{\longrightarrow} \mathcal I_C (t) \rightarrow \mathcal I_{Z|H}(t) \rightarrow 0
\]
which yields the long exact sequence
{\footnotesize
  \begin{equation} \label{std exact} 0 \rightarrow [I_C]_{t-1}
    \stackrel{\times L}{\longrightarrow} [I_C]_t \rightarrow
    [I_{Z|H}]_t \rightarrow [A]_{t-1} \stackrel{\times
      L}{\longrightarrow} [A]_t \rightarrow H^1(\mathcal
    I_{Z|H}(t)) \rightarrow H^2(\mathcal I_C(t-1)) \rightarrow
    H^2(\mathcal I_C(t)) \rightarrow 0.
  \end{equation}}
Since the Hilbert function of $C$ is determined, the question of
whether $A$ has the WLP depends completely on understanding the
Hilbert function of $I_{Z|H}$.

%

We will make use of the following observation.

\begin{lemma} \label{inj}
For $t < 2d$, the map $\times L : [A]_{t-1} \rightarrow [A]_t$ is injective if and only if $[I_{Z|H}]_t = 0$. In particular, $A$ has the WLP if and only if $[I_{Z|H}]_{2d-2} = 0$.
\end{lemma}

\begin{proof}
We note the following facts:

\begin{itemize}
\item[(i)] for $t<2d$ we have $\dim [I_C]_t = 0$ (Lemma \ref{basic facts about C} (c)). 

\item[(ii)] $R/I$ has no socle until the last non-zero degree (it is a complete intersection) and is self-dual as a graded module, so to prove that $R/I$ has the WLP it is enough to prove injectivity of $[A]_{2d-3} \rightarrow [A]_{2d-2}$ (see for instance \cite{MMN}, Proposition 2.1).

\end{itemize}
Then the result follows from the above exact sequence, setting $t = 2d-2$.
\end{proof}

We now recall some notation and results from \cite{HU}, which we state
in our setting.

\begin{definition}
  Let $M$ be a finitely generated graded $R$-module.  We set
  \[
  a_- (M) = \min \{ i | \ [M]_i \neq 0 \}.
  \]
\end{definition}

\begin{lemma}[\cite{HU}, Socle Lemma] \label{socle lemma} Let $M$ be a
  nonzero finitely generated graded $R$-module.  Let $L \in [R]_1$
  be a general linear form and let
  \[
  0 \rightarrow K \rightarrow M(-1) \stackrel{\times
    L}{\longrightarrow} M \rightarrow D \rightarrow 0
  \]
  be exact.  If $K \neq 0$ then $a_- (K) > a_- (\Soc (D))$.
\end{lemma}

\begin{remark} \label{soc of h1} Assume the following: $S = k[x,y,z]$, $Z$ is a
  zero-dimensional subscheme of $\mathbb P^2 = H$ with homogeneous ideal
  $I_Z \subset S$, $B = S/I_Z$, $\bar B$ is an artinian reduction of
  $B$ by a linear form, and
  \[
  0 \rightarrow \bigoplus_{i=1}^{b_2} S(-n_{2,i}) \rightarrow
  \bigoplus_{i=1}^{b_1} S(-n_{1,i}) \rightarrow S \rightarrow B
  \rightarrow 0
  \]
  is the minimal free resolution.  In \cite{HU} it is also pointed out
  that then
  \[
  \Soc (H^1_\ast(\mathcal I_Z)) = \bigoplus_{i=1}^{b_2} k(-n_{2,i} +3) =
  \Soc (\bar B)(1).
\]
\end{remark}

Still following the work of \cite{HU}, let $M = A = R/I$ and let $C$
be as above.  From (\ref{std exact}) we see that $\Soc (D) \subset
\Soc (H^1_\ast(\mathcal I_{Z|H}))$.  We thus obtain
\[
a_- (K) > a_- (\Soc (D)) \geq a_- (\Soc (H^1_\ast (\mathcal I_{Z|H})))
\geq a_- (\Soc (\bar B)) -1;
\]
that is
\begin{equation} \label{hu ineq} a_-(K) \geq a_- (\Soc (\bar B)).
\end{equation}

\begin{remark} \label{trans socle lemma}
Notice that $K$ represents the forms in $I_{Z|H}$ that do not lift to $I_C$. Thus one way of phrasing the result of the Socle Lemma, which we will use, is that

\begin{quotation}

{\it the socle of $\bar B$ starts no later than the least degree of a form vanishing on $Z$ that does not lift to $C$.}

\end{quotation}

\noindent Since the latter degree can sometimes be read from the Hilbert function of $\bar B$, this can be used to force socle elements in $\bar B$.
\end{remark}

  For a finite (reduced) set
of points $Z$ in projective space $\mathbb P^n$, the first difference
of the Hilbert function is a finite sequence of positive integers,
also known as the {\em $h$-vector} of $Z$.  When $Z$ is the general
hyperplane section of   a reduced, irreducible curve $C$, it was shown by
Harris (cf. \cite{harris2}, \cite{EH}) that $Z$ has the so-called {\em
  uniform position property} (UPP); \label{UPP def} that is, any two subsets of $Z$ of
the same cardinality have the same Hilbert function.  If the
irreducible curve $C$ lies in $\mathbb P^3$, we may view its general
hyperplane section $Z$ as lying in a plane $\mathbb P^2$.  Harris
notes that in this case UPP implies that the $h$-vector of $Z$ is of
{\em decreasing type}, meaning that once the values experience a
strict decrease, they are strictly decreasing until they reach zero.

The following useful result of Davis \cite{davis} should be viewed as
an extension of this result.  Recall that the \emph{regularity}, or
\emph{Castelnuovo-Mumford regularity} of $Z$ agrees with the top
degree of the $h$-vector of $Z$.

\begin{theorem}[Davis \cite{davis}] \label{davis thm} Let $\{ h_i \ |
  \ 0 \leq i \leq k \}$ be the $h$-vector of a reduced zero-dimensional subscheme $Z$ in $\mathbb P^2$.  Suppose that
  $h_{t-1} = h_{t} > 0$ for some $t \leq k$.  Then
  \begin{enumerate}
  \item The elements of the homogeneous components $[I_Z]_{t-1}$ and $[I_Z]_t$ all
    have a common factor, $F$, of degree equal to $h_t$, which thus is a common factor for all components
    in degree $< t-1$ as well.
  \item $F$ is reduced.
  \item Let $Z_1$ be the subset of $Z$ consisting of all the points
    lying on the curve $F$.  Then $Z_1$ has $h$-vector $\{ g_i \ | \ 0
    \leq i \leq k \}$ where
    \[
    g_i = \min \{ h_i, h_t \}.
    \]
    In particular, the regularity of $Z_1$ is the same as that of $Z$.
  \item Let $Z_2$ be the subset of $Z$ consisting of all points of $Z$
    that do not lie on $F$.  Then the $h$-vector of $Z_2$ is $\{ f_i \
    | \ 0 \leq i \leq m \}$ where
    \[
    \begin{array}{rcl}
      m & = &  (t-2) - h_t \\
      f_i & = & h_{h_t +i} - g_{h_t+i}
    \end{array}
    \]
  \end{enumerate}
\end{theorem}

\begin{example}
  Suppose that $Z$ has $h$-vector $(1,2,3,4,5,6,7,8,5,3,3,2,1)$.  Then
  $t = 10$, $h_t = 3$, and we compute the $h$-vectors of $Z_1$ and
  $Z_2$ as follows (note that the $h$-vector for $Z_2$ displayed below is shifted by 3).
  \[
  \begin{array}{c|cccccccccccccccccc}
    Z 	& 1 	& 2 	& 3 	& 4 	& 5 	& 6 	& 7 	& 8 	& 5 	& 3 	& 3 	& 2 	& 1 \\
    Z_1 	& 1 	& 2 	& 3 	& 3 	& 3 	& 3 	& 3 	& 3 	& 3 	& 3 	& 3 	& 2 	& 1 \\ \hline
    Z_2 	& 	&	&	& 1 	& 2 	& 3 	& 4 	& 5 	& 2
  \end{array}
  \]
  \noindent In particular, $Z$ contains a subset, $Z_1$, of $33$
  points on a reduced cubic curve and a subset, $Z_2$, of $17$ points
  that do not lie on the cubic.
\end{example}


\section{Measuring failure of WLP} \label{measuring failure}

\begin{notation} \label{def CI(d,d,d,d)}
Let $d \geq 2$ be an integer. We denote by $CI(d,d,d,d)$ the space of ideals in $R = k[x_1,x_2,x_3,x_4]$ generated by a regular sequence of four forms of degree $d$. We view $CI(d,d,d,d)$ as a dense open subset of the Grassmannian $Gr(4,N)$, where $N = \binom{d+3}{3}$.
\end{notation}

Let $I \in CI(d,d,d,d)$. Let $A=R/I$ and let $h_{A}$ be the Hilbert function of $A$. We note that the socle degree of $A$ (the last degree in which $h_{A}$ is non-zero) is $4d-4$, that $h_{A}$ is symmetric, and that the maximum value of $h_{A}$ occurs exactly in degree $2d-2$. Also, $A$ is self-dual (after a twist).

\begin{remark}
We will use Hilbert functions to measure the failure of $A$ to satisfy the WLP.  Note first that all $I \in CI(d,d,d,d)$ give rise to algebras with the same Hilbert function, since their Betti diagrams come from the Koszul sequence and so are identical. Furthermore, since $4d-4$ is even and all the generators have the same degree $d$, the Hilbert function in degree $2d-2$ is strictly greater than the Hilbert function in any other degree.
\end{remark}

\begin{remark} \label{duality}
Since $A=R/I$ is, in particular, a Gorenstein algebra, it is self-dual as a graded module, and as noted above its last non-zero component is in degree $4d-4$. Hence  for any $L \in [R]_1$ and any $i \geq 0$, by duality the rank of the homomorphism $\times L : [A]_{2d-3-i} \rightarrow [A]_{2d-2-i}$ is the same as the rank of the corresponding homomorphism from degree $2d-2+i$ to degree $2d-1+i$.
\end{remark}

The next lemma shows that in order to check whether $A=R/I$ has the WLP, it is enough to check surjectivity in just one degree (and injectivity for the first half follows automatically).

\begin{lemma} \label{one spot}
$A$ has the WLP if and only if $\times L : [A]_{2d-2} \rightarrow [A]_{2d-1}$ is surjective.
\end{lemma}

\begin{proof}
Since the peak of the Hilbert function occurs in degree $2d-2$, one direction is trivial. The reverse implication follows from the exact sequence
\[
\cdots\rightarrow [A]_{t-1} \stackrel{\times L}{\longrightarrow} [A]_t \rightarrow [A/(L)]_t \rightarrow 0
\]
and the fact that once the cokernel is zero in one degree, it is zero in all subsequent degrees.
\end{proof}

Next we show that if $A$ fails to be surjective by the smallest amount possible in the degree mentioned in Lemma \ref{one spot} then it must actually be surjective in all subsequent degrees (and by duality, injectivity only fails in one degree in the first half, and it is by the smallest possible amount). 

\begin{lemma} \label{others have max rk}
 If $\times L : [A]_{2d-2} \rightarrow [A]_{2d-1}$ has a one-dimensional cokernel then,   in all subsequent degrees, $\times L$ is surjective. Consequently, failure of injectivity in the first half is also by the smallest amount possible.
\end{lemma}

\begin{proof}
Looking at the Hilbert function of the cokernel of $\times L : A (-1) \rightarrow A$, namely $A/(L)$, we have assumed that $h_{A/(L)}(2d-1) = 1$. Then Macaulay's theorem (\cite{BH} Theorem 4.2.10 (c)) gives that $h_{A/(L)}(t) \leq 1$ for all $t \geq 2d$. If $h_{A/(L)}(2d) = 1$ then this is maximal growth from degree $2d-1$ to degree $2d$. 

We claim that this prevents $I$ from being artinian in degree $2d$. Indeed, the fact that $A/(L)$ has Hilbert function with maximal growth from degree $2d-1$ to degree $2d$ forces the ideal $(I,L)$ to have no minimal generator in degree $2d$, since otherwise removing such a generator gives an ideal $J$ whose Hilbert function exceeds Macaulay's bound. Let $J$ be the ideal generated by $(I,L)_{\leq 2d}$. Then by the Gotzmann Persistence Theorem (\cite{BH} Theorem 4.3.3), the value of the Hilbert function of $R/J$ in all degrees $\geq 2d-1$ remains 1, so in degrees $\geq 2d-1$, $J$ is the saturated ideal of a single point, hence not artinian. Thus also $(I, L)$ fails to be artinian in degrees $2d-1$ and $2d$, and so in particular it is not artinian in degree $2d$. The second part follows by duality.
\end{proof}

This motivates the following definition.

\begin{definition}
Let $I \in CI(d,d,d,d)$ and let $L$ be a general linear form.

\begin{itemize}

\item We say that $A$ {\it fails WLP by one} if   $\times L : [A]_{2d-2} \rightarrow [A]_{2d-1}$ has a one-dimensional cokernel.

\item We say that $A$ {\it fails WLP by more than one} if $\times L : [A]_{2d-2} \rightarrow [A]_{2d-1}$ has a cokernel that is more than one-dimensional.

\end{itemize}
\end{definition}

\begin{remark}
These properties can be read immediately from the Hilbert function of $A/(L)$:

\begin{itemize}
\item  if $h_{A/(L)} (2d-1) = 0$ then $A$ has WLP.

\item If $h_{A/(L)} (2d-1) = 1$ then $A$ fails WLP by one.

\item If $h_{A/(L)} (2d-1) \geq 2$ then $A$ fails by more than one.
\end{itemize}
\end{remark}

We do not yet know that for all $I \in CI(d,d,d,d)$, $A$ has the WLP. Nevertheless, for fixed $I$ there is an open subset of linear forms for which the cokernel of $\times L$ has the smallest Hilbert function in degree $2d-1$ among all $L \in (\mathbb P^3)^*$.

\begin{lemma}
Fix $I \in CI(d,d,d,d)$. Let
\[
m = \min_{L \in [R]_1} \{ h_{A/(L)}(2d-1) \}.
\]
(Note $m=0$ if and only if $A$ has the WLP.) Let
\[
U_I = \{ L \in (\mathbb P^3)^* \ | \ h_{A/(L)} (2d-1) = m \} .
\]
Then $U_I$ is open in $(\mathbb P^3)^*$.
\end{lemma}

\begin{proof}
Fix a basis for $[A]_{2d-2}$ and one for $[A]_{2d-1}$, and say the two dimensions are $s_1$ and $s_2$. Then for  a  linear form $L = a_1 x_1 + a_2 x_2 + a_3 x_3 + a_4 x_4$, the multiplication
\[
\times L : [A]_{2d-2} \rightarrow [A]_{2d-1}
\]
can be represented by a $s_2 \times s_1$ matrix $M$ of linear forms in the dual variables $a_1, a_2, a_3, a_4$. The  maximal minors of $M$ may or may not all be zero. Let $s$ be the largest integer for which there is an $s \times s$ minor of $M$ that is not identically zero. The vanishing locus of the ideal generated by the $s \times s$ minors gives a closed subscheme of $(\mathbb P^3)^*$, and the open complement represents the linear forms giving multiplication of maximum possible rank, i.e. $U_I$.
\end{proof}

\begin{remark} Given $I \in CI(d,d,d,d)$ and $L \in [R]_1$, there are two reasons why $\times L : [A]_{2d-2} \rightarrow [A]_{2d-1}$ can fail to have maximal rank. One is that $A$ might fail to have the WLP. The other is that $A$ does have the WLP but $L$ is not general enough.
\end{remark}

\begin{definition}
We will say that $L$ is {\it general for $I$} (or for $A$) if $L \in U_I$.
\end{definition}


\section{Computations on the hyperplane section of $C$}
\label{curve section}

We maintain the notation from Section \ref{sect tools}, and we first note some useful facts:
\begin{itemize}
\item $\dim [I_C]_t = 0$ for $t \leq 2d-1$.
\item $\dim [I_C]_{2d} = 4$.
\item $h^2 (\mathcal I_C (t)) = 0 $ for $t \geq 2d-3$ (from the Koszul
  resolution for $I_{C_i}$ and the fact that
  \[
  H^2(\mathcal I_C(t)) \cong H^1(\mathcal O_C(t)) \cong H^1(\mathcal
  O_{C_1}(t)) \oplus H^1(\mathcal O_{C_2}(t)) \cong H^2(\mathcal
  I_{C_1}(t)) \oplus H^2(\mathcal I_{C_2}(t))
  \]
  since $C$ is a disjoint union).
\end{itemize}

\begin{notation}

Since $C$ is  a curve in $\mathbb P^3$, and a hyperplane $H$
is a plane, the ideal of the hyperplane section $I_{Z|H}$ can be
viewed as the homogeneous ideal of a finite set of points in $\mathbb
P^2$.  As already mentioned, we denote the coordinate ring of $H$ by
$S \cong k[x,y,z]$, and since there is no chance of confusion, we will denote the ideal of $Z$ as $I_Z$ instead of $I_{Z|H}$.

In the work so far we used the notation $B = S/I_Z$, and $\overline B$ for the general artinian reduction of $B$, consistent with \cite{HU}. However, now the geometry of $Z$ will play a greater role, so while we continue to use $B$ for the coordinate ring of $Z$, we will use the more suggestive notation $h_Z$ for the Hilbert function.
  \end{notation}

  We first describe the Hilbert function of $B =
S/I_{Z}$ under the assumption that $A = R/\langle
F_1,F_2,F_3,F_4\rangle$ has the WLP; we shall call this the {\em
  expected Hilbert function} for~$B$.

\begin{lemma} \label{exp hf} The algebra $A$ has the WLP if and only if the expected
  $h$-vector of $B = S/I_{Z}$, which is the first difference
  of the expected Hilbert function of $B$, is generic, namely it is
  \[
  \begin{array}{c|cccccccccccccccc}
    \text{degree} 	& 0 	& 1 	& 2 	& \dots 	& 2d-3 	& 2d-2 	& 2d-1 	& 2d \\ \hline
    \Delta h_Z 	& 1 	& 2 	& 3 	& \dots 	& 2d-2 	& 2d-1 	& d 	& 0.
  \end{array}
  \]
\end{lemma}

\begin{proof}
  Notice that the sum of the entries of the claimed $h$-vector is
  $2d^2 = \deg C$, as required.  Observe that the socle degree of $A$
  is $4d-4$, so $h_A$ is strictly increasing until degree
  $\frac{4d-4}{2} = 2d-2$, and then it is strictly decreasing.  Recall also that
  $I_C = I_1 \cdot I_2$.

  Now consider the long exact sequence (\ref{std exact}) and set $t = 2d-2$. We know that $A$ has the WLP
  if and only if $\times L : [A]_{2d-3} \rightarrow [A]_{2d-2}$ is injective. Since $[I_C]_{2d-2} = 0$, we get that $A$ has the WLP if and only if $[I_Z]_{2d-2} = 0$. Furthermore, from the Koszul resolution we see from an easy computation that
  \[
  \dim [A]_{2d-2} - \dim [A]_{2d-3} = d.
  \]
  Hence (by duality and looking at surjectivity) $A$ has the WLP if and only if $\dim \ker (\times L) = d$ from degree $2d-2$ to degree $2d-1$. Then again from (\ref{std exact}) we get that $A$ has the WLP if and only if $\dim [I_Z]_{2d-1} = d$, which (after a trivial calculation) completes the proof.
\end{proof}

Whether or not $A$ has the WLP, the first difference of the Hilbert function of $Z$ has
a nice ``see-saw" behavior, in that if $\Delta h_Z$ falls below the
predicted value by a fixed amount in a given degree before degree
$2d-1$, then it is above the predicted value (0) by the same amount in
the corresponding degree after $2d-1$.  In particular, the value in
degree $2d-1$ is $d$ regardless of the extent to which WLP fails to
hold.

\begin{lemma} \label{seesaw} Whether or not $A$ has the WLP, suppose
  that $\Delta h_Z (2d-1-m) = 2d-m-c$. Then $\Delta h_Z(2d-1+m) = c$.
  That is, for $0 \leq m \leq 2d-1$, we have
  \[
  \Delta h_Z (2d-1-m) + \Delta h_Z(2d-1+m) = 2d-m.
  \]
  In particular, $\Delta h_Z(2d-1) = d$.
\end{lemma}

\begin{proof}
  We have $I_C = I_{C_1} \cap I_{C_2} = \langle F_1, F_2 \rangle \cap
  \langle F_3, F_4 \rangle$.  Since $\langle F_1 F_3, F_2 F_4 \rangle
  $ is a regular sequence, it links $C$ to the curve $D$ with $I_D =
  \langle F_1, F_4 \rangle \cap \langle F_2, F_3 \rangle$, which is
  again a union of two complete intersections of the same degree, with
  the same Hartshorne-Rao module $M(D) = R/\langle F_1, F_2, F_3, F_4
  \rangle $ (note that this module is self-dual up to twist).  For a
  general hyperplane $H$, the set of points $Z$ cut out on $C$ is
  linked on $H$ to the set of points $Y$ cut out on $D$ by a complete
  intersection of type $(2d, 2d)$, and clearly $Z$ and $Y$ have the
  same Hilbert function since $C$ and $D$ have the same Hilbert
  function and the same Hartshorne-Rao module.  The assertion of the
  lemma then comes immediately from the formula for the behavior of
  the Hilbert function under linkage (see \cite{DGO}, Theorem 3).
\end{proof}

\begin{lemma} \label{hf for fail by one}
The complete intersection $A$ fails WLP by one if and only if the first difference of the Hilbert function of $k[x,y,z]/I_Z$ is
  \[
  \begin{array}{c|cccccccccccccccc}
    \text{degree} 	& 0 	& 1 	& 2 	& \dots 	& 2d-3 	& 2d-2 	& 2d-1 	& 2d \\ \hline
    \Delta h_Z 	& 1 	& 2 	& 3 	& \dots 	& 2d-2 	& 2d-2 	& d 	& 1
  \end{array}
  \]
\end{lemma}

\begin{proof}
It follows immediately from the arguments in Lemma \ref{others have max rk}, Lemma \ref{exp hf},   together with the exact sequence (\ref{std exact}).
\end{proof}

Following Lemma \ref{socle lemma}, let $K$ be the kernel of the multiplication $A(-1) \stackrel{\times L}{\longrightarrow} A$. Recall that for a graded module $M$ we defined $a_- (M) $ to be the degree of the first non-zero component of $M$.

\begin{corollary} \label{a-(K)}
For $t < 2d$ we have
\[
 \dim [K]_t - \dim [K]_{t-1} = t+1 - \Delta h_Z (t).
 \]
 In particular,
$\displaystyle a_{-} (K) = \min \{ t \ | \ \Delta h_Z (t) < t+1 \}$.

\end{corollary}

\begin{proof}
For any fixed $t$, consider the long exact sequence (\ref{std exact}).
We know that $[I_C]_t = 0$ for $t < 2d$. On the other hand, Lemma \ref{seesaw} guarantees that $\Delta h_Z(2d-1) < 2d$, so in the range $0 \leq t \leq 2d-1$ we have $\dim [K]_t = \dim [I_{Z}]_t$. Hence in this range
\[
\dim [K]_t - \dim [K]_{t-1} = \dim [I_{Z}]_t - \dim [I_{Z}]_{t-1} = t+1-\Delta h_Z(t)
\]
as claimed.
\end{proof}

\begin{example} \label{1st ex} Suppose $d=6$.  Then the Hilbert function of $R/I = M(C)$ is given by the following table:
  \[
  \arraycolsep=4pt
  \begin{array}{rrrrrrrrrrrrrrrrrrrrrrrrrrrrrrrrrr}
    0 & 1 & 2 & 3 & 4 & 5 & 6 & 7 & 8 & 9 & 10 & 11 & 12 & 13 & 14 &
    15 & 16 & 17 & 18 & 19 & 20 \\ \hline
    1 & 4 & 10 & 20 & 35 & 56 & 80 & 104 & 125 & 140 & 146 & 140 & 125
    & 104 & 80 & 56 & 35 & 20 & 10 & 4 & 1
  \end{array}
  \]
  For a general choice of $I$, we know from Lemma \ref{exp hf} that $B
  = k[x,y,z]/I_{Z}$ has $h$-vector (i.e. first difference of the
  Hilbert function) equal to
  \begin{equation}
    \begin{array}{c|rrrrrrrrrrrrrrrrrrrrrrrrrrrrrrrrr}
      \text{degree} & 0 & 1 & 2 & 3 & 4 & 5 & 6 & 7 & 8 & 9 & 10 & 11 & 12  \\ \hline
      \Delta h_Z & 1 & 2 & 3 & 4 & 5 & 6 & 7 & 8 & 9 & 10 & 11 & 6 & 0
    \end{array}
  \end{equation}
  if the WLP holds. Suppose that the multiplication on $A$ from degree 7 to degree 8
  fails to be injective.  Then the maps from degree 8 to degree 9 and
  from degree 9 to degree 10 also fail to be injective, and the
  kernels are isomorphic to the components of $I_{Z}$ in degrees 8,
  9 and 10 respectively.  The preceding lemmas allow the following
  $h$-vector for $k[x,y,z]/I_{Z}$:
  \begin{equation} \label{bad hf}
    \begin{array}{c|rrrrrrrrrrrrrrrrrrrrrrrrrrrrrrrrr}
      \text{degree} & 0 & 1 & 2 & 3 & 4 & 5 & 6 & 7 & 8 & 9 & 10 & 11 & 12 & 13 & 14   \\ \hline
      \Delta h_Z& 1 & 2 & 3 & 4 & 5 & 6 & 7 & 8 & 7 & 7 & 6 & 6 & 5 & 3 & 2
    \end{array}
  \end{equation}
  \noindent but not the following one:
  \[
  \begin{array}{c|rrrrrrrrrrrrrrrrrrrrrrrrrrrrrrrrr}
    \text{degree} & 0 & 1 & 2 & 3 & 4 & 5 & 6 & 7 & 8 & 9 &10 & 11 & 12 & 13 & 14  \\ \hline
    \Delta h_Z & 1 & 2 & 3 & 4 & 5 & 6 & 7 & 8 & 7 & 7 & 6 & 6 & 5 & 4 & 1
  \end{array}
  \]
  Two methods will play an important role in our proofs: the symmetry principle (see (\ref{symm princ}) and the Cayley-Bacharach theorem. To illustrate them
     suppose that (\ref{bad hf}) were to occur.  By Davis' theorem
  (Theorem \ref{davis thm}), $Z$ contains a subset of 71 points on a
  curve $F$ of degree 7, and one point, $P$, not lying on $F$.  But
  $C_1$ and $C_2$ are smooth curves of degree $36$, whose ideals were
  chosen as general 2-dimensional subspaces of the 4-dimensional space
  of sextics spanned by $F_1,F_2,F_3,F_4$.  Let $Z_1$ and $Z_2$ be the
  hyperplane sections of $C_1$ and $C_2$, respectively, with the
  general plane $H$.  Both $Z_1$ and $Z_2$ have the uniform position
  property (UPP).  By symmetry (since $I_1$ and $I_2$ are chosen
  generally, so are indistinguishable), there cannot be one
  distinguished point of this sort, since $P$ would have to lie on
  either $Z_1$ or $Z_2$.  Furthermore, to illustrate another of our
  tools, any curve of degree 7 containing all but one point of $Z_1$
  (resp.\ $Z_2$) must contain all of $Z_1$ (resp.\ $Z_2$) by the
  Cayley-Bacharach property of complete intersections.
\end{example}

\begin{lemma} \label{rule out} An $h$-vector of the form
  \[
  \begin{array}{c|ccccccccccccccccccccccccccccccccc}
    \text{degree} & 0 	& 1 	& 2 	& 3 	& \dots 	& t-1 	& t 	& t+1 	& \dots \\ \hline
    \Delta h_Z & 1 	& 2 	& 3 	& 4 	& \dots 	& t 	& t 	&  a 	&   \dots
  \end{array}
  \]
  \noindent does not occur for the coordinate ring, $B$, of a general
  hyperplane section $Z$ of $C$, for $a=t$ or $t-1$.
\end{lemma}

\begin{proof}
  We maintain the notation of Lemma \ref{socle lemma}, Remark
  \ref{soc of h1} and Remark \ref{trans socle lemma}. We see from the $h$-vector that the kernel $K$ begins in
  degree $t$ (Corollary \ref{a-(K)}).  We thus know from (\ref{hu ineq}) (or from Remark \ref{trans socle lemma}) that $\bar B$ has
  socle in degree $t-1$ or $t$.  Since $I_{Z}$ has only one minimal generator $F$ of degree $\le t$, $\bar B$ coincides with the
  hypersurface ring of $F$ in degrees $\leq t$, which has no socle, so the socle degree of $\bar B$ must be in degree $t$.

  Suppose that $a=t$.  Then in fact $\bar B$ is a hypersurface ring (the corresponding ideal has only one generator)  in degrees $\leq t+1$, so we have a contradiction
  with the socle in degree $t-1$ or $t$.

  Suppose that $a=t-1$. Then $I_{Z}$ has one minimal generator $F$
  in degree $t$ and one minimal generator $G$ in degree $t+1$. Call $\bar F$ and $\bar G$ the corresponding elements in the artinian reduction.

  We first claim that
  these two generators have a syzygy of the form $QF + LG = 0$ with
  $\deg Q = 2$ and $\deg L = 1$. Indeed, since there is a socle element in degree $t$,  the minimal free resolution of $I_{Z}$ has a term $ R(-t-2)$,
 which represents syzygies of all generators of degrees $\leq t+1$, of which there are only $F$ and $G$.

  Thus $F$ and $G$ have a common
  factor of degree $t-1=a$.  By Davis' theorem, $Z$ has a subset $Z'$
  with $h$-vector
  \[
  \begin{array}{c|ccccccccccccccccccccccccccccccccc}
    \text{degree} 	& 0 	& 1 	& 2 	& 3 	& \dots 	& t-2 	& t-1 	&t 	& t+1 	&  \dots \\ \hline
    \Delta h_Z 	& 1 	& 2 	& 3 	& 4 	& \dots 	& t-1 	& t-1 	&  t-1 	& t-1 	&  \dots
  \end{array}
  \]
  \noindent where the rest of the $h$-vector agrees with that of $B$.
  Thus there are two points of $Z$ not lying on the curve defined by
  this common factor.  As in Example \ref{1st ex}, by symmetry, one
  point comes from $Z_1$ and one comes from $Z_2$.  By assumption and
  Lemma \ref{seesaw}, we must have $t+1 < 2d-1$, so $t -1 < 2d-3$.
  But then by the Cayley-Bacharach theorem, any curve of degree $t-1$
  containing all but one of the points of the complete intersection
  $Z_1$ contains all of $Z_1$, giving a contradiction.
\end{proof}

Even though it is not true that $Z$ has the UPP (see page \pageref{UPP def}), it is still true that the Hilbert function of $Z$ has the decreasing type property:

\begin{proposition} \label{decreasing type}
  The $h$-vector of the general hyperplane section of $C$ is of
  decreasing type.
\end{proposition}

\begin{proof}
  Suppose that the $h$-vector is not of decreasing type.  By
  definition, this means that for some $s$ and $t$ we have
  \[
  0 < s = \Delta h_Z(t-1) = \Delta h_Z(t) < t.
  \]
  From Lemma \ref{seesaw}, a calculation gives that we must have $t \leq 2d-1$. Indeed, recall
  that $\Delta h_Z$ is itself an $O$-sequence, and
  if $t \geq 2d$ then Lemma \ref{seesaw} forces a growth of the Hilbert function $\Delta h_Z$ in
  degree $\leq 2d-2$ that exceeds Macaulay's bound. Hence
  $s \geq d$, again by Lemma \ref{seesaw}.  By Theorem \ref{davis thm},
  $Z$ contains a subset $Z'$ lying on a curve $D$ of degree $s$ and
  having $h$-vector
  \[
  \begin{array}{c|ccccccccccccccccccccccccccccccccc}
    \text{degree} & 0 & 1 & 2 & 3 & \dots & s-2 & s-1 & s & \dots &  t-1 &t & t+1 &  \dots \\ \hline
    \Delta h_{Z'} & 1 & 2 & 3 & 4 & \dots & s-1 & s  & s & \dots & s & s & h_{t+1} &  \dots
  \end{array}
  \]
  We will first bound above the number of points of $Z$ that lie off
  $D$, thus bounding below the number of points of $Z'$.  By the
  assumption that the $h$-vector is not of decreasing type, the number
  of points off $D$ is not zero.  At least half of these must be
  points of either $Z_1$ or $Z_2$ (and in fact by symmetry it must be
  exactly half).  Without loss of generality, say it is $Z_1$.  Since
  $C_1$ is a smooth irreducible curve, $Z_1$ satisfies the UPP.  We
  will use this fact to get a contradiction.  We will use the fact
  that the $h$-vector of $Z_1$ is
  \[
  \begin{array}{c|ccccccccccccccccccccccccccccccccc}
    \text{degree} & 0 & 1 & 2 & 3 & \dots & d-2 & d-1 & d  & \dots &  2d-3 & 2d-2 &  \dots \\ \hline
    \Delta h_{Z_1} & 1 & 2 & 3 & 4 & \dots & d-1 & d  & d-1 & \dots & 2 & 1
  \end{array}
  \]
  \bigskip
  \noindent \underline{Case 1}: If $t = 2d-1$ then $s=d$ by Lemma
  \ref{seesaw}.  Then the number of points of $Z$ not lying on $D$ is
  at most
  \[
  1 + 2 + \cdots + (d-2) = \binom{d-1}{2}
  \]
  so the number of points of $Z$ on $D$ is at least $2d^2 -
  \binom{d-1}{2}$.  Hence $Z_1$ contains at least $d^2 - \frac{1}{2}
  \binom{d-1}{2}$ points on $D$.  On the other hand, the sum of the
  entries of the $h$-vector of $Z_1$ through degree $s = d = \deg D$
  is
  \[
  1 + 2 + 3 + \dots + d + (d-1) = \binom{d+1}{2} +(d-1) \leq d^2 -
  \frac{1}{2} \binom{d-1}{2}
  \]
  as long as $d \geq 2$.  Since $Z_1$ has the UPP, this means that any
  curve of degree $d$ containing at least $d^2 - \frac{1}{2}
  \binom{d-1}{2}$ points of $Z_1$ must contain all of $Z_1$.  This
  contradicts the fact that some points of $Z$, hence of $Z_1$, do not
  lie on $D$.

  \bigskip
  \noindent \underline{Case 2}: If $t \leq 2d-2$ then the number of
  points of $Z$ not lying on $D$ is at most
  \[
  1+2+3 + \dots + (t-1-s) = \binom{t-s}{2}
  \]
  (since the value of the $h$-vector of $Z$ is greater than $s$ at
  most up to degree $t-2$, and the value there is at most $t-1$).
  Hence the number of points of $Z$ that do lie on $D$ is at least
  $2d^2 - \binom{t-s}{2}$, so $Z_1$ contains at least $ d^2 -
  \frac{1}{2} \binom{t-s}{2}$ points on $D$.  Note that we have $s \geq d$,
  and in fact by Case 1 we can assume $s > d$.
 As before, the sum of
  the entries of the $h$-vector of $Z_1$ through degree $s$ is
  \[
  1+2+3+ \dots + (d-1) + d + (d-1) + (d-2) + \dots + (2d-1-s) =
  \binom{d+1}{2} + \binom{d}{2} - \binom{2d-1-s}{2}.
  \]
  We claim that
  \begin{equation} \label{claim} d^2 - \frac{1}{2} \binom{t-s}{2} \geq
    \binom{d+1}{2} + \binom{d}{2} - \binom{2d-1-s}{2}.
  \end{equation}
  As in case 1, UPP for $Z_1$ will then mean that any curve of degree
  $s$ containing at least $d^2 - \frac{1}{2} \binom{t-s}{2}$ points of
  $Z_1$ contains all of $Z_1$, which gives the desired contradiction.
  But (\ref{claim}) is equivalent to
  \[
  (2d-1-s)(2d-2-s) \geq \frac{(t-s)(t-s-1)}{2}
  \]
  which is clearly true.
\end{proof}

Next we give a result for any $d > 1$, for any complete intersection in four variables  which is less than full WLP but improves the range where we know maximal rank of the multiplication by a general linear form to hold.

\begin{theorem} \label{d1=d2=d3=d4} Let $A = R/I = R/\langle
  F_1,F_2,F_3,F_4 \rangle$ where $I$ is a complete intersection and
  $\deg F_i = d$ for all $i$.  Let $L$ be a general linear form.
  Then the multiplication maps $\times L \colon
  [A]_{t-1}\rightarrow[A]_t$ are injective for $t<
  \lfloor\frac{3d+1}{2}\rfloor$.
\end{theorem}

\begin{proof}
  We have to determine the $h$-vector for $Z$ with the smallest
  possible initial degree, given the constraints in our lemmas.
  Suppose first that $d$ is even.  Then the smallest possible initial
  degree comes when the $h$-vector has the form {\small
    \[
    (1,2,3,\dots, b-1, b, b-1,\dots,d+1,d,\dots) \hbox{ \ or \ }
    (1,2,3,\dots,b-1,b,b,b-2,\dots, d+1,d,\dots),
    \]} \hspace{-.35cm} where $d$ occurs in degree $2d-1$. In both cases  the  initial degree is $b$.  One checks that this initial
  degree is $\frac{3d}{2}$.  When $d$ is odd, the smallest possible
  initial degree comes, for instance, when the $h$-vector has the form
  \[
  (1,2,3,\dots, b-1, b, b-2,\dots,d+1,d,\dots)
  \]
  thanks to Lemma \ref{rule out}.  One checks that the initial degree
  is $\frac{3d+1}{2}$.
\end{proof}


\section{WLP for small $d$} \label{small d}

In this section we prove the WLP for complete intersections of type $(d,d,d,d)$ for $d \leq 5$.  The cases $3 \leq d \leq 5$ were previously open. In the next section we  outline a possible approach for all $d$ and give two conjectures that, if proved, would allow to demonstrate WLP following this approach. We maintain the notation of Remark \ref{soc of h1}; in particular, $S = k[x,y,z]$ is (by slight abuse of notation) the coordinate ring for the plane $H$ containing the points $Z$. Furthermore, $B = S/I_Z$ and $\bar B$ is an artinian reduction of $B$ by a linear form, say $L'$. We will denote by $\bar I_Z$ the ideal $\frac{I_Z + (L)'}{(L')}$ in $T = S/(L')$, so $\bar B =T / \bar I_Z$, and recall again that the graded Betti numbers of $\bar I_Z$ over $T$ are the same as those of $I_Z$ over $S$. In a similar way, for other sets of points in $H$ we will denote by $\bar \ $ the ideal of the artinian reduction of those coordinate rings.

\begin{example} \label{d=2}
When $d=2$, it was shown in \cite{MN-quadrics} Corollary 4.4 that a complete intersection of four quadrics has the WLP.
\end{example}

\begin{proposition} \label{d=3}
A complete intersection of four cubics has the  WLP.
\end{proposition}

\begin{proof}
We give two easy arguments based on our results.
Let $I$ be a complete intersection of four cubics, and let $Z$ be the corresponding set of 18 points.
Now $Z$ is the union of two complete intersections of type $(3,3)$. The Hilbert function of $R/I$ is
  \[
  \begin{array}{c|cccccccccccccccc}
    \text{degree} 	& 0 	& 1 	& 2  & 3 & 4 & 5 & 6 &  7 & 8 & 9  \\ \hline
    h_{R/I} 	& 1 	& 4 & 10 & 16 & 19 & 16  & 10 & 4 & 1 & 0
  \end{array}
  \]
and the $h$-vector of $Z$ has the form
  \[
  \begin{array}{c|cccccccccccccccc}
    \text{degree} 	& 0 	& 1 	& 2 	& 3 & 4 & 5 & 6 \\ \hline
    \Delta h_Z 	& 1 	& 2 	& 3 	& 4 & x & 3 & 5-x & 0
  \end{array}  \ .
  \]
Then $x$ cannot be 4 by Lemma \ref{rule out}, $x$ cannot be 3 by Proposition  \ref{decreasing type}, and it cannot be $< 3$ since then the $h$-vector would not be an $O$-sequence. Hence $x=5$ and $R/I$ has the WLP by Lemma \ref{exp hf}.

Alternatively, one could simply apply Theorem \ref{d1=d2=d3=d4} for the case $d=3$ to get injectivity from degree 3 to 4, which is enough to prove WLP.
\end{proof}

\begin{theorem} \label{d=4}
A complete intersection of four forms of degree 4 has the WLP.
\end{theorem}

\begin{proof}

Let $A = R/I = R/\langle  F_1,F_2,F_3,F_4 \rangle$ where $I$ is a complete intersection and
  $\deg F_i = 4$ for all $i$.  Then we claim that  $A$ has the WLP.
   The Hilbert function of $R/I$ is
  \[
  \begin{array}{c|cccccccccccccccc}
    \text{degree} 	& 0 	& 1 	& 2  & 3 & 4 & 5 & 6 &  7 & 8 & 9 & 10 & 11 & 12 & 13  \\ \hline
    h_{R/I} 	& 1 	& 4 & 10 & 20 & 31 & 40 & 44 & 40 & 31 & 20 & 10 & 4 & 1 & 0
  \end{array}
  \]
The midpoint of the Hilbert function of $R/I$  is in degree 6, so we expect injectivity for $\times L : [A]_{t-1} \rightarrow [A]_t$ for all $t \leq 6$.

Note that $Z = Z_1 \cup Z_2$ is the general hyperplane section of $C = C_1 \cup C_2$, so both $Z_1$ and $Z_2$ are complete intersections of type $(4,4)$.
By Theorem \ref{d1=d2=d3=d4}, $\times L : [A]_{t-1} \rightarrow [A]_t$ is injective for $t < 6$. Thus $[K]_t = 0$ for $t < 6$, so by Corollary \ref{a-(K)} we have $\Delta h_Z(5) = 6$. As a result, the $h$-vector of $Z$ has the form
  \[
  \begin{array}{c|cccccccccccccccc}
    \text{degree} 	& 0 	& 1 	& 2 	& 3 & 4 & 5 & 6 & 7 & 8 & 9 & 10  \\ \hline
    \Delta h_Z 	& 1 	& 2 	& 3 	& 4 & 5 & 6 & 7-c & 4 &  c & 0 & 0
  \end{array}
  \]
  (using Lemma \ref{seesaw}).

We want to show $c=0$. Our results give us that $c$ can only be 1 or 2.
That is, we have to rule out $(1,2,3,4,5,6,5,4,2)$ and $(1,2,3,4,5,6,6,4,1)$. Notice that there is a minimal generator for $I_Z$ in degree 6 if and only if $A$ fails to have WLP, so we want to show that such  a generator cannot exist.

  We first rule out $c=2$. Suppose $\Delta h_Z$ is given by
 \[
  \begin{array}{c|ccccccccccccccccccccccccccccccccc}
    \text{degree}  & 0 	& 1 	& 2 	& 3 	& 4 & 5 & 6 & 7 & 8 & 9 \\ \hline
    \Delta h_Z 	& 1 	& 2 	& 3 	& 4 	& 5 & 6 & 5 & 4 & 2 & 0
  \end{array}
  \]
In particular, $I_Z$ has two minimal generators in degree 6 and possibly one in degree 7 (among possibly others), and $\bar B$ has socle in degree $\leq 6$ (using (\ref{hu ineq}) and Corollary \ref{a-(K)}). Clearly it is not in degree $\leq 4$.

\begin{itemize}

\item If $\bar B$ has socle in degree 5 then  the last free module in a minimal free resolution of $I_{Z}$, or equivalently $\bar{B}$, over $S$ has a direct summand $S (-7)$. So the two generators of $I_Z$ of degree 6 have a linear syzygy: $L_1 F_1 + L_2 F_2 = 0$.  This means that they have a common factor of degree 5. By Davis's theorem, all the points of $Z$ but one lie on a curve of degree 5. But this violates the symmetry of the situation: from the 4-dimensional vector space $[I]_4$ we chose two general two-dimensional subspaces to define the two curves $C_1$ and $C_2$, and we are looking at a general hyperplane section  of their union. It is impossible that one of these two curves is distinguished by this  property of containing the single point not on the curve of degree 5.

\item Now assume $\bar B$ has socle beginning in degree 6 and assume that $I_{Z}$ has two minimal generators, $F_1$ and $F_2$, of degree 6 but none of degree 7. Then we have a quadratic syzygy $Q_1 F_1 + Q_2 F_2 = 0$. This means that $F_1$ and $F_2$ have a common factor, $G$, of degree 4. Since $I_{Z}$ has no minimal generator in degree 7, $G$ is also common to the entire component in degree 7. By Davis's theorem, since $\Delta h_Z(7)=4$, all but 4 of the points of $Z$ lie on a curve of degree 4. By symmetry, two have to come from $Z_1$ and two from $Z_2$. By liaison and UPP, any curve of degree 4 containing 14 of the 16 points of $Z_1$ or of $Z_2$ contains all 16 points. Thus we have a contradiction.

\item Finally, assume that $\bar B$ has socle beginning in degree 6 and assume that $I_{Z}$ has two minimal generators, $F_1$ and $F_2$, of degree 6  and a minimal generator $G$ of degree 7. Then $S/\langle F_1, F_2 \rangle$ has Hilbert function with first difference $(1,2,3,4,5,6,5,5, \dots)$, and so Davis's theorem implies that all but one of the points lies on a curve of degree 5. Again this is impossible because of the symmetry principle.

\end{itemize}

Now we will rule out the case $c=1$. Thus, we have to show that the $h$-vector of $Z$ cannot be $(1,2,3,4,5,6,6,4,1)$.

Suppose the $h$-vector of $Z$ is $(1,2,3,4,5,6,6,4,1)$.
We first consider the socle of $\bar B$. Thanks to Remark \ref{trans socle lemma}, there is socle in degree $\leq 6$, and from the $h$-vector it can only be in degree exactly 6. Since this is the Hilbert function of an artinian algebra over $k[x,y]$, the canonical module must have exactly one generator in its initial degree and at least two generators in the second degree. This means that the socle is exactly 1-dimensional in degree 8 and at least 2-dimensional in degree 7. By Lemma \ref{socleB}  we will see that in degree $2d-1$ there must be exactly a $(d-2)$-dimensional socle, so in this case it must be exactly 2-dimensional (and not 3-dimensional, even though this is numerically possible); we will assume that here.

Hence the last free module in the minimal free resolution of $I_{Z}$ must at least have free summands $S(-8), S(-9)^2, S(-10)$. The only ambiguity is the possibility of a second summand $S(-8)$, so we will indicate this with the exponent $1+\epsilon \ (\epsilon \geq 0)$.

What about generators? Certainly there is exactly one generator in degree 6 and exactly two in degree 7. We must have  five or six minimal generators (depending on $\epsilon$), and considering degrees we see that the only possibility for the resolution is
  \begin{equation} \label{mfrZ}
  0 \rightarrow
  \begin{array}{c}
 S(-8)^{1+\epsilon} \\
  \oplus \\
S(-9)^2 \\
  \oplus \\
  S(-10)
  \end{array}
  \rightarrow
  \begin{array}{c}
  S(-6) \\
  \oplus \\
  S(-7)^2 \\
  \oplus \\
S(-8)^{2+\epsilon}
  \end{array}
  \rightarrow  I_{Z} \rightarrow 0.
  \end{equation}
  And remember that $Z$ has $h$-vector
  \[
  (1,2,3,4,5,6,6,4,1).
  \]

We will use an argument that analyzes how the Betti numbers change via a sequence of two links,
\[
Z := Y_0 \sim Y_1 \sim Y_2
\]
which we will now describe. A key idea is that each link will be viewed simultaneously as a single link and as a pair of two separate links.
{\it  The rest of the proof is going to rely on symmetry. We will perform a series of links in which we treat $Z_1$ and $Z_2$ equally, and come to a situation where one of them has different behavior than the other.} This will allow us to exclude the case $c=1$ in view of the Symmetry Principle (\ref{symm princ}).

We start with $Z = Y_0 =  Y_{0,1} \cup Y_{0,2}$, where $Y_{0,1}$ and $Y_{0,2}$ are disjoint complete intersections of type $(4,4)$ and $Z$ is the general hyperplane section of $C = C_1 \cup C_2$. We saw above that $I_Z$ has one minimal generator of degree 6, two of degree 7 and at least two of degree 8.

Both $Y_{0,1}$ and $Y_{0,2}$ are general hyperplane sections of smooth curves, so they both have UPP. If the minimal generator of $I_{Y_0}$ of degree 6 were reducible, then either it consists of a cubic containing $Y_{0,1}$ and one containing $Y_{0,2}$ (which is clearly impossible since both are complete intersections of quartics) or else there are distinguished subsets of $Y_{0,1}$ and $Y_{0,2}$ consisting of those points lying on the various factors. But this violates uniformity. So without loss of generality we can assume that the sextic generator is irreducible. Furthermore, again by uniformity, no point of  $Y_0$ is a singular point of the sextic.

The first link will be by a complete intersection of type $(6,8)$, where the curve of degree 8 is the union of a general quartic containing $Y_{0,1}$ and a general quartic containing $Y_{0,2}$. Let $F$ be the sextic generator. $Z$ is a set of $32 = 16 + 16$ smooth points of $F$. The base locus of the linear system on $F$ of quartics containing $Y_{0,1}$ is just $Y_{0,1}$, so the residual to $Y_{0,1}$ in a general element of this linear system is a set of reduced points $Y_{1,1}$ on $F$. Similarly we get a set of reduced points $Y_{1,2}$ on $F$, and $Y_{1,1} \cap Y_{1,2} = \emptyset$ by generality. Since in both cases the quartic is a minimal generator of $I_{Y_{0,i}}$, one checks that both $Y_{1,1}$ and $Y_{1,2}$ are complete intersections of type $(2,4)$.  Setting $Y_1 = Y_{1,1} \cup Y_{1,2}$, then $Y_1$ is reduced and is linked to $Y_0 = Z$. In summary,
\[
Y_{0,1} \cup Y_{0,2} = Y_0 \sim Y_1 = Y_{1,1} \cup Y_{1,2} \hbox{ \ where \ } Y_{0,1} \sim Y_{1,1} \hbox{ and } Y_{0,2} \sim Y_{1,2}.
\]
A calculation gives that $Y_1$ has $h$-vector $(1,2,3,4,4,2)$ and free resolution (after splitting two  terms in the mapping cone coming from minimal generators used in the link)
\[
0 \rightarrow
\begin{array}{c}
S(-6)^{1+\epsilon} \\
\oplus \\
S(-7)^2
\end{array}
\rightarrow
\begin{array}{c}
S(-4) \\
\oplus \\
S(-5)^2 \\
\oplus \\
S(-6)^{1+\epsilon}
\end{array}
\rightarrow I_{Y_1} \rightarrow 0.
\]
Note that if $I_{Y_1}$ had two minimal generators of degree 6 then the ideal $\bar I$ generated by the quartic and two quintics would have a common factor of degree 2. By Davis's theorem, the subscheme of $Y_1$ not lying on this conic has $h$-vector $(1,2,2)$, and so $Y_{1,1}$ and $Y_{1,2}$ behave differently: one of them contributes 2 points to this residual and the other contributes 3 points. This violates symmetry, so we conclude that $\epsilon =0$.

The second link will use curves in $I_{Y_1}$ of degrees $4,5$, where the quartic is the union of the conic containing $Y_{1,1}$ and the conic containing $Y_{1,2}$, and the residual, $Y_2$, has degree $20 - 16 = 4$ and is similarly the union of two complete intersections  of type $(1,2)$, so
\[
Y_{1,1} \cup Y_{1,2} = Y_1 \sim Y_2 = Y_{2,1} \cup Y_{2,2} \hbox{ \ where \ } Y_{1,1} \sim Y_{2,1} \hbox{ and } Y_{1,2} \sim Y_{2,2}.
\]
We first justify the existence of this link. We have seen that $I_{Y_1}$ has one minimal generator, say $G$, of degree 4 (which is the union of a conic containing $Y_{1,1}$ and a conic containing $Y_{1,2}$) and two minimal generators of degree 5. Suppose that a general choice, $H$, of quintic has a component in common with $G$. This can only happen if $H$ shares a line with each of the two conics, by symmetry (since neither $Y_{1,1}$ nor $Y_{1,2}$ can play a different role from the other). So $G$ is actually the union of four lines, with four points of $Y_1$ on each.

Now we revisit the first link. We have an irreducible curve $F$ of degree 6 containing both $Y_{0,1}$ and $Y_{0,2}$ and we have two parallel links. Consider the pencils $|4H_F - Y_{0,1}|$ and $|4H_F - Y_{1,1}|$ of divisors on $F$, where $H_F$ is the divisor on $F$ cut out by a hyperplane (a line). $Y_1$ is the union of a general element of the first pencil and a general element of the second pencil, and as we have seen, these elements are both complete intersections of type $(2,4)$. By Bertini, it is impossible for every element of both of these pencils to lie on a reducible conic. Thus the second link also exists.

What can we say about $Y_2$? It has degree 4 and is the union of a complete intersection $Y_{2,1}$ of type $(1,2)$ linked to $Y_{1,1}$ and a complete intersection $Y_{2,2}$ of type $(1,2)$ linked to $Y_{1,2}$ as above. Its $h$-vector is $(1,2,1)$. Furthermore, there is enough choice in the links so that $Y_2$ is reduced.

Next one computes that the minimal free resolution of $Y_2$ is
\[
0 \rightarrow
\begin{array}{c}
S(-4) \\
\oplus \\
S(-3)
\end{array}
\rightarrow
\begin{array}{c}
S(-3) \\
\oplus \\
S(-2)^2
\end{array}
\rightarrow I_{Y_2} \rightarrow 0.
\]
The key observation is that there can be no splitting of the $S(-3)$ in both free modules. Indeed, this resolution comes from the {\it minimal} free resolution for $I_{Y_1}$ given above, which has a redundant $S(-6)$, and the link is done with a complete intersection, say $X$, of type $(4,5)$. The generators for $I_{Y_1}$ have degrees 4, 5 and 6. Then the mapping cone relating the resolutions of $I_{Y_1}$, $I_X$ and $I_{Y_2}$ allows only splitting of the summands $S(-4)$ and $S(-5)$ from the resolution for $I_{Y_1}$. In particular, the summands $S(-6)$ occurring (redundantly) in the minimal free resolution of $I_{Y_1}$ survive to the resolution for $I_{Y_2}$, but with the twisting they both become $S(-3)$ and neither is split off.

But this minimal free resolution (with the repeated $S(-3)$) defines the union of a point and three  additional collinear points. The former point is distinguished in this set of four points. But it must arise either from the series of links $Y_{0,1}  \sim Y_{1,1} \sim Y_{2,1}$ or the series of links $Y_{0,2}  \sim Y_{1,2} \sim Y_{2,2}$ . By the symmetry of this construction, such a special, distinguished, point is impossible.
\end{proof}

Since our main result is not quite to the point of proving WLP, we also give a complete proof for the case $d=5$ even though there are no new ideas beyond the $d=4$ case.

 \begin{corollary} \label{d=5}
 Let $A = R/I = R/\langle  F_1,F_2,F_3,F_4 \rangle$ where $I$ is a complete intersection and
   $\deg F_i = 5$ for all $i$.  Then $A$ has the WLP.
 \end{corollary}

 \begin{proof}
 The argument is very similar to that of Theorem \ref{d=4}, with a few minor differences. Now the midpoint of the Hilbert function of $R/I$ is in degree 8, and we expect injectivity for $\times L : [A]_{t-1} \rightarrow [A]_t$ for all $t \leq 8$. In this case Theorem \ref{d1=d2=d3=d4} gives it to us for $t \leq 7$, so we only have to prove it for $t=8$. Now we get that $\Delta h_Z$ must have the form
  \[
   \begin{array}{c|ccccccccccccccccccccccccccccccccc}
     \text{degree}  & 0 	& 1 	& 2 	& 3 	& 4 & 5 & 6 & 7 & 8 & 9 & 10 & 11 \\ \hline
     \Delta h_Z 	& 1 	& 2 	& 3 	& 4 	& 5 & 6 & 7 & 8 & 9-c & 5 & c & 0
   \end{array}
   \]
 and we must have $0 \leq c \leq 3$ to preserve decreasing type.  Again we are trying to show that $c=0$, and to show this we are supposing $c >0$ in order to obtain a contradiction. And this means that there is an unexpected generator in degree 8, which we will use and which will lead to the contradiction.

Note first that for such $c$ there is a  $c$-dimensional vector space of forms of degree 8 containing $Z$ (which is the smallest such degree).
 As in the proof of Theorem~\ref{d=4}, the general such element is reduced and irreducible and smooth at the points of $Z$, since (a) both $Z_1$ and $Z_2$ have UPP and lie on no curve of degree $\leq 4$, and (b) by symmetry neither can behave in a way different from the other.

Since we will again have to follow a sequence of  parallel links, we maintain the notation that $Z = Y_0 = Y_{0,1} \cup Y_{0,2}$, a disjoint union of complete intersections of type $(5,5)$. We link using a general element of degree 8 (which is irreducible) and the union of two forms of degree 5, being general elements of $[I_{Y_{0,1}}]_5$ and $[I_{Y_{0,2}}]_5$.

 \vspace{.1in}

 \noindent \underline{Case 1}: $c=3$.

 In this case the $h$-vector of $Y_0$ is
 \[
  \begin{array}{c|ccccccccccccccccccccccccccccccccc}
    \text{degree}  & 0 	& 1 	& 2 	& 3 	& 4 & 5 & 6 & 7 & 8 & 9 & 10 & 11 \\ \hline
    \Delta h_{Y_0} 	& 1 	& 2 	& 3 	& 4 	& 5 & 6 & 7 & 8 & 6 & 5 & 3
  \end{array}
  \]
  and the $h$-vector of the residual $Y_1$ is $(1,2,3,4,5,6,4,3,2)$. Because of irreducibility and the abundance of choices for the link, we can assume that $Y_1$ lies on an irreducible sextic. We note that $Y_1$ is the union of two complete intersections of type $(3,5)$. Now we link $Y_1$ using two sextics, one of which is irreducible and the other is the union of the cubic containing $Y_{1,1}$ and the cubic containing $Y_{1,2}$. The residual has $h$-vector $(1,2,1,1,1)$, meaning that there are exactly five points on a line and one point not on the line. But this violates symmetry since the distinguished point must either come originally from $Y_{0,1}$ or from $Y_{0,2}$. Contradiction.

   \vspace{.1in}

 \noindent \underline{Case 2}: $c=2$.

Now the $h$-vector of $Y_0$ is
 \[
  \begin{array}{c|ccccccccccccccccccccccccccccccccc}
    \text{degree}  & 0 	& 1 	& 2 	& 3 	& 4 & 5 & 6 & 7 & 8 & 9 & 10 & 11 \\ \hline
    \Delta h_{Y_0} 	& 1 	& 2 	& 3 	& 4 	& 5 & 6 & 7 & 8 & 7 & 5 & 2
  \end{array}
  \]
  and the $h$-vector of the residual $Y_1$ is $(1,2,3,4,5,6,5,3,1)$. But note that $Y_1$ is still the union of two complete intersections of type $(3,5)$. Linking again using two sextics, one of which is irreducible and the other is the union of the cubic containing $Y_{1,1}$ and the cubic containing $Y_{1,2}$, we obtain a residual that is the union of two complete intersections of type $(1,3)$ and has $h$-vector $(1,2,2,1)$.

Thanks to Lemma \ref{socle lemma}, $\bar B$ has socle in degree $\leq 8$. It cannot be in degree $\leq 6$. If there were socle in degree 7 then the minimal free resolution would have a summand $S(-9)$ in the second free module, meaning that the two generators of degree 8 have a linear syzygy and hence a common factor. But we know they form a regular sequence, so this is impossible. Hence there is a summand $S(-10)^{1+\alpha}$ in the second free module in the resolution for $I_{Y_0}$.

Up to twist, the artinian reduction (whose Hilbert function is the $h$-vector of $Y_0$) has dual given by its canonical module, and the socle elements of the artinian reduction correspond to generators for the canonical module. Thus there is  a 2-dimensional socle in degree 10 (the initial degree of the canonical module, up to twist). Since the canonical module is a module over a polynomial ring in two variables, the two generators in minimal degree span at most a 4-dimensional component in the second degree. Thus we deduce from the equality $\Delta h_{Y_0}=5$  that there is at least a 1-dimensional socle in degree 9. Turning to minimal generators, in addition to the ones mentioned above there may be generators in degrees 10 and 11. So the minimal free resolution has the shape
\[
0 \rightarrow
\begin{array}{c}
S(-10)^{1+\alpha} \\
\oplus \\
S(-11)^{1 + \beta} \\
\oplus \\
S(-12)^2
\end{array}
\rightarrow
\begin{array}{c}
S(-8)^2 \\
\oplus \\
S(-9) \\
\oplus \\
S(-10)^\gamma \\
\oplus \\
S(-11)^\delta
\end{array}
\rightarrow I_{Y_0} \rightarrow 0.
\]
Since the twists have to add to the same thing in both free modules, we obtain the equation
\[
20 = 10(\gamma-\alpha) + 11(\delta-\beta).
\]
We conclude that $\beta = \delta$ and $\gamma = 2+ \alpha $.
Our first link uses a generator of degree 8 and one of degree 10 to get a residual $Y_1$ with $h$-vector $(1,2,3,4,5,6,5,3,1)$.  Computing the mapping cone and splitting the summands $S(-8)$ and $S(-10)$ we obtain for the residual $Y_1$ the minimal free resolution
\[
0 \rightarrow
\begin{array}{c}
S(-7)^\beta \\
\oplus \\
S(-8)^{1+\alpha} \\
\oplus \\
S(-9) \\
\oplus \\
S(-10)
\end{array}
\rightarrow
\begin{array}{c}
S(-6)^2 \\
\oplus \\
S(-7)^{1+\beta} \\
\oplus \\
S(-8)^{1+\alpha}
\end{array}
\rightarrow I_{Y_1} \rightarrow 0.
\]
(In particular there is a redundant $S(-8)$ that does not split off.) As in the case $c=3$, $Y_1 = Y_{1,1} \cup Y_{1,2}$, the union of two complete intersections of type $(3,5)$.

Note that to obtain $Y_{1,1}$ and $Y_{1,2}$ we found the residual to $Y_{0,1}, Y_{0,2}$ respectively on the irreducible curve of degree 8 cut out by the two pencils of quintics, each of which has as base locus the sets $Y_{0,1}, Y_{0,2}$. So $Y_{1,1}$ and $Y_{1,2}$ have UPP and in particular $Y_1$ lies in a complete intersection of type $(6,6)$. The residual, $Y_2$, has $h$-vector $(1,2,2,1)$ and is the union of two complete intersections of type $(1,3)$. Its minimal free resolution has the form
\[
0 \rightarrow
\begin{array}{c}
S(-4)^{1+\alpha}  \\
\oplus \\
S(-5)^{1+\beta} \\
\end{array}
\rightarrow
\begin{array}{c}
S(-2) \\
\oplus \\
S(-3) \\
\oplus \\
S(-4)^{1+\alpha} \\
\oplus \\
S(-5)^{\beta}
\end{array}
\rightarrow I_{Y_2} \rightarrow 0.
\]
From this we see immediately that $\alpha = \beta = 0$ since the $h$-vector is $(1,2,2,1)$. But the fact that there is a generator in degree 4 forces four points on a line, which is impossible since if two points from each of $Y_{2,1}$ and $Y_{2,2}$ lie on this line then all of $Y_2$ lies on the line, and otherwise we have a violation of the symmetry principle.

This concludes the case $c=2$.

 \vspace{.1in}

 \noindent \underline{Case 3}: $c=1$.

 We now show that $c=1$ is impossible. Now the $h$-vector $h_Z$ is
 \[
 (1,2,3,4,5,6,7,8,8,5,1).
 \]
 There is at least one socle element in degree 8 (and none earlier), at least three in degree 9 (considering the canonical module of $\bar B$ as in the computation of the socle for the artinian reduction of $S/I_{Y_0}$ in Case 2), and exactly one in degree 10. In terms of minimal generators for $I_Z$, we have  exactly one in degree 8 and exactly three in degree 9. There may be some in degree 10, but as before (using UPP of the two subsets) we can rule out any of degree 11. So far, to build the minimal free resolution, we have
  \[
 0 \rightarrow
 \begin{array}{c}
 S(-10)^{1+\epsilon_1} \\
 \oplus \\
 S(-11)^{3+\epsilon_2} \\
 \oplus \\
 S(-12)
 \end{array}
 \rightarrow
 \begin{array}{c}
 S(-8) \\
 \oplus \\
 S(-9)^3 \\
 \oplus \\
 S(-10)^{\epsilon_3}
 \end{array}
 \rightarrow  I_Z \rightarrow 0.
 \]
 Since the sum of the twists of the first free module must equal the sum of the twists of the second one, we get
 \[
 55 + 10 \epsilon_1 + 11 \epsilon_2 = 35 + 10 \epsilon_3, \ \hbox{ i.e. } \ 20 = 10(\epsilon_3 - \epsilon_1) - 11 \epsilon_2.
 \]
 This forces $\epsilon_2 = 0$ and $\epsilon_3 - \epsilon_1 = 2$ (so in particular $\epsilon_3 \geq 2$), i.e. setting $\epsilon := \epsilon_1$ we have the minimal free resolution is
 \[
 0 \rightarrow
 \begin{array}{c}
 S(-10)^{1+\epsilon} \\
 \oplus \\
 S(-11)^3 \\
 \oplus \\
 S(-12)
 \end{array}
 \rightarrow
 \begin{array}{c}
 S(-8) \\
 \oplus \\
 S(-9)^3 \\
 \oplus \\
 S(-10)^{2+\epsilon}
 \end{array}
 \rightarrow  I_Z \rightarrow 0.
 \]
 (To prove that $\epsilon_2=0$ we could also have invoked the argument at the beginning of section \ref{strategy}, as we did in Theorem \ref{d=4}, but it seemed simpler to use the numerical argument above.)

 Now we play the same game as above.  We have $Z = Y_0 = Y_{0,1} \cup Y_{0,2}$ and we link $Y_0 \sim Y_1$ where the octic is irreducible and the 10-ic is the union of a general element of degree 5 in $[I_{Y_{0,1}}]_5$ and a general element of degree 5 in $[I_{Y_{0,2}}]_5$.
This  links $Y_0$ to a residual $Y_1$ that is the union of two complete intersections, $Y_{1,1}$ and $Y_{1,2}$, of type $(3,5)$ and has $h$-vector $ (1,2,3,4,5,6,6,3)$ and minimal free resolution
  \[
 0 \rightarrow
 \begin{array}{c}
 S(-9)^3 \\
 \oplus \\
 S(-8)^{1+\epsilon} \\
 \end{array}
 \rightarrow
 \begin{array}{c}
 S(-8)^{1+\epsilon} \\
 \oplus \\
 S(-7)^3 \\
 \oplus \\
 S(-6)
 \end{array}
 \rightarrow  I_{Y_1} \rightarrow 0
 \]
 where the redundant $S(-8)$ does not split. We also know that  the generator of degree 6 is the product of the cubics in $I_{Y_{1,1}}$ and $I_{Y_{1,2}}$.

 Now we link $Y_1$ to $Y_2$ using the above-mentioned sextic and a general element of degree 7 in $I_{Y_1}$ (which will be irreducible). The residual, $Y_2$, has $h$-vector $(1,2,3,4,2)$ and is the union of two complete intersections of type $(2,3)$. Its minimal free resolution is
   \[
 0 \rightarrow
 \begin{array}{c}
 S(-6)^2 \\
 \oplus \\
 S(-5)^{1+\epsilon} \\
 \end{array}
 \rightarrow
 \begin{array}{c}
 S(-5)^{1+\epsilon} \\
 \oplus \\
 S(-4)^3
 \end{array}
 \rightarrow  I_{Y_2} \rightarrow 0
 \]
We note in passing that we must have $\epsilon=0$. Indeed, from the Hilbert function we see $\epsilon \leq 1$. If $\epsilon=1$ then there are two minimal generators in degree 5. Setting $J$ to be the ideal generated by the three quartics in the artinian reduction, the Hilbert function of $\bar R/J$ is $(1,2,3,4,2,2,\dots)$ which forces the quartics to have a common factor of degree 2. By Davis's theorem, there are three points of $Y_2$ not on this conic, which violates symmetry.

We would like to link $Y_2$ using two quartics, and we see now that this can only fail to be possible if the component in degree 4 has a common factor of degree 1. We again set $J$ to be the ideal generated by the three quartics in the artinian reduction. Then $\bar R/J$ has Hilbert function $(1,2,3,4,2,1,1, \dots)$. Since $\bar B$ has Hilbert function $(1,2,3,4,2)$, we conclude that $Z$ has  five points on the linear common factor. Again this is impossible by symmetry.

We conclude that we can link  using two quartics, one of which is the obvious union of two conics, to get a residual $Y_3$ that is the union of two complete intersections of type $(1,2)$ and has $h$-vector $(1,2,1)$ and minimal free resolution
 \[
 0 \rightarrow
 \begin{array}{c}
 S(-4) \\
 \oplus \\
 S(-3) \\
 \end{array}
 \rightarrow
 \begin{array}{c}
 S(-3) \\
 \oplus \\
 S(-2)^2 \\
 \end{array}
 \rightarrow  I_{Y_3} \rightarrow 0
 \]
This forces $Y_3$ to have a subscheme of degree 3 lying on a line, and again we have a violation of the symmetry principle.
 \end{proof}


\section{A strategy for degree 6 and higher} \label{strategy}

The results of sections \ref{measuring failure} and \ref{curve section} give strong restrictions on the possible Hilbert function of the general hyperplane section of $C$ and consequently on the possible behavior of $R/I$ from the point of view of the WLP. However, we also saw in section \ref{small d} that as $d$ increases, the cases that one has to check to prove WLP become overwhelming. Even for $d=4$ and 5 it was complicated, and we did not push beyond that point.

Nevertheless, there is strong indication that this approach will work for arbitrary $d$. In this section we give a detailed outline of how a proof should proceed, refining the approach that we used for $d=4,5$.
Unfortunately, we are forced to leave open two points on which the whole proof will ultimately rest, which we will label as formal conjectures. The first is a reduction step, and the second is a technical point that we have not yet resolved.

The first conjecture is that it is enough to check the case where WLP fails by one.

\begin{conjecture} \label{force exactly one}

If there is a complete intersection $J \in  CI(d,d,d,d)$ (see Definition \ref{def CI(d,d,d,d)}) that fails WLP by more than one then there also exists a complete intersection $I \in CI(d,d,d,d)$ that fails WLP by exactly one.

More precisely, we conjecture that the locus $X \subset CI(d,d,d,d)$ of complete intersections that fail by more than one is contained in the closure of the locus $Y$ of complete intersections that fail by exactly one.

\end{conjecture}

 Of course the idea is to show that neither $J$ nor $I$ actually exist, i.e. that $X$ and $Y$ are empty.
Here is the difficulty to proving our conjecture. We start with the complete intersection $J$ and choose a linear form $L$ that is general for $J$. It is not too hard to show that then there is a complete intersection $I$ for which $\times L$ fails WLP by one. But it is not clear why $L$ is also general for $I$.

\begin{quotation}
{\it
Assuming Conjecture \ref{force exactly one}, we can assume now without loss of generality that $R/I$ fails the WLP by exactly one, and seek a contradiction.}
\end{quotation}

\noindent It is important to note that the cases $d=4,5$ did not make such an assumption; those proofs are complete.

As before we denote by $Z $ the general hyperplane section of the union $C$ of two complete intersection curves $C_1, C_2$ obtained by taking two general 2-dimensional subspaces of the 4-dimensional vector space spanned by $F_1, F_2, F_3, F_4$. Of course $Z$ is the union of two complete intersection sets of points in $\mathbb P^2$.

By Lemma \ref{hf for fail by one}, the stated failure of injectivity translates to a Hilbert function $\Delta h_Z$ as follows.
 \[
  \begin{array}{c|ccccccccccccccccccccccccccccccccc}
    \text{degree}  & 0 	& 1 	& 2 	& 3 	& \dots & 2d-3 & 2d-2 & 2d-1& 2d & 2d+1 \\ \hline
    \Delta h_Z	& 1 	& 2 	& 3 	& 4 	& \dots & 2d-2 & 2d-2 & d & 1 & 0.
  \end{array}
  \]
As in Corollary \ref{d=5}, we first describe the minimal free resolution of $I_Z$, and we start by considering the socle of the artinian reduction of $B = S/I_Z$.  It is clear that there is a one-dimensional socle in degree $2d$. In degree $2d-1$ there is either a $(d-2)$-dimensional socle or a $(d-1)$-dimensional socle.

\begin{lemma}\label{socleB}
The artinian reduction of $B = S/I_Z$ has a  $(d-2)$-dimensional socle.
\end{lemma}
\begin{proof}
 Suppose the dimension of the socle in degree $2d-1$ were $(d-1)$-dimensional. Call $\bar B$ the artinian reduction of $B$, whose Hilbert function is $\Delta h_Z$. Quotienting out its $(d-1)$-dimensional socle in degree $2d-1$, we get an algebra $k[x,y]/J$ whose Hilbert function in degree $2d-1$ and $2d$ is equal to 1. Thus $[J]_{2d}$ has a linear gcd, which lifts to a gcd of $I_Z$. A theorem of Davis (\cite{davis}, Theorem 4.1; see also \cite{BGM}, Theorem 2.4) then gives that $Z$ contains $2d+1$ points on a  line, which is absurd since $Z$ is the union of two sets of points with the uniform position property. This proves that the dimension of the socle in degree $2d-1$ is $d-2$, as claimed.
\end{proof}

Finally, as before, there has to be non-zero socle in degree $2d-2$ since there is a curve of degree $2d-2$ containing $Z$ that does not lift to $C$. It is not clear yet what the dimension of the socle in this degree is.

 For generators, we have exactly one in degree $2d-2$ and $d-2$ in degree $2d-1$. There is none in degree $2d+1$. We have some in degree $2d$ that we have to compute. After a computation with the twists as we did in Theorem \ref{d=4}, we get a minimal free resolution of the form
  \begin{equation} \label{Z-d}
  0 \rightarrow
  \begin{array}{c}
S(-2d-2) \\
  \oplus \\
S(-2d-1)^{d-2} \\
  \oplus \\
  S(-2d)^{1+\epsilon}
  \end{array}
  \rightarrow
  \begin{array}{c}
  S(-2d+2) \\
  \oplus \\
  S(-2d+1)^{d-2} \\
  \oplus \\
S(-2d)^{2+\epsilon}
  \end{array}
  \rightarrow  I_{Z} \rightarrow 0
  \end{equation}
where $\epsilon \geq 0$.
The idea of our approach will rest on the fact that the redundancy in the minimal free resolution (in this case $S(-2d)$) is preserved in the sequence of linked sets of points that we will produce. More precisely,
the strategy will be to study a series of $d-2$ specific links, and obtain a contradiction after the last link (or earlier).

We start with the set $Z$ which is the disjoint union of two complete intersections, $Z_1$ and $Z_2$ of type $(d,d)$. The result of each  link will again be a scheme-theoretic union of two complete intersections of the same type (but we do not claim a priori that they are reduced or disjoint in general).

Setting the notation, let $Z = Y_0 = Y_{0,1} \cup Y_{0,2}$; then for $i = 1, \dots, d-2$, the $i$-th link will send $Y_{i-1} = Y_{i-1,1} \cup Y_{i-1,2}$ to $Y_i = Y_{i,1} \cup Y_{i,2}$, where $Y_{i,1}$ and $Y_{i,2}$ are  complete intersections of the same type.
More precisely, the $i$-th link will start with $Y_{i-1} = Y_{i-1,1} \cup Y_{i-1,2}$  and will consist of a regular sequence $(F_i,G_i)$ where $F_i$ is a form of some degree, say $a_i$, in $I_{Y_{i-1}}$ and $G_i = G_{i,1} G_{i,2}$ is  the product of two forms of the same degree, say $b_i$, one in $I_{Y_{i-1,1}}$ and one in $I_{Y_{i-1,2}}$. To represent such a link numerically, we will use the notation $(a_i, b_i+b_i)$.

We consider such a sequence of links both as linking $Y_{i-1}$ to $Y_i$, and  the same time as representing two ``parallel" sequences of links, where we view  $(F_i,G_{i,1})$ as a link from $Y_{i-1,1}$ to $Y_{i,1}$  and $(F_i,G_{i,2})$ as a link from $Y_{i-1,2}$ to $Y_{i,2}$ separately. Thus we have
\[
Y_{i-1} = Y_{i-1,1} \cup Y_{i-1,2} \ \overset{(a_i, b_i+b_i)}{\scalebox{3.5}[1]{$\sim$}} \ Y_i = Y_{i,1} \cup Y_{i,2}.
\]
It is crucial to note that we start with $Z_1$ and $Z_2$ that are indistinguishable both geometrically and numerically, and at each step the choices of the links $(F_i, G_{i,1}G_{i,2})$  do not ``favor" either $Y_{i,1}$ over $Y_{i,2}$ or vice versa. Thus we have the

\vspace{.1in}

\begin{equation} \label{symm princ}
\parbox{5.7in}  {\noindent \bf {Symmetry Principle} \it{For each choice of $i$, there can be no geometric or numerical property distinguishing $Y_{i,1}$ from $Y_{i,2}$.}
}
\end{equation}

\vspace{.1in}

For the sake of clarity, our argument will proceed as follows:

\begin{itemize}

\item[(i)] describe numerically the sequence of links that we will use,

\item[(ii)] look at the $h$-vectors, assuming such links exist,

\item[(iii)]  look at the resolutions, again assuming that the links exist,

\item[(iv)] justify the existence of the links, and finally

\item[(v)] put it all together for the conclusion.

\end{itemize}

\noindent The goal is to obtain a contradiction. We stress that for  steps (i), (ii) and (iii) we will focus on the numerical information of the desired links and residuals, and will discuss the existence of these links only in  step (iv). Once the existence of these links is established, the fact that the residuals are unions of complete intersections of certain types is a routine calculation obtained by looking at the ``parallel" links separately (e.g.  for the first link, a complete intersection of type $(d,d)$ is linked by a complete intersection of type $(d,2d-2)$ to a complete intersection of type $(d,d-2)$).

\vspace{.1in}

\noindent \underline{Step (i)}:  For each link the following describes the starting set $Y_{i-1} = Y_{i-1,1} \cup Y_{i-1,2}$ as a union of complete intersections, then gives the degrees of the generators of the complete intersection giving the link, and finally describes the residual as a union of complete intersections. There is a total of $d-2$ links.

\begin{center}

\scriptsize{
\begin{tabular}{c|c|c|ccccccccccccc}
link & starting points & link & residual points \\
$\#$ & $Y_{i-1}$ in $\mathbb P^2$  & type  & $Y_i$ in $\mathbb P^2$ \\ \hline
1 & $Y_0 = CI(d,d) \cup CI(d,d)$ & $(d+d, 2d-2)$ & $Y_1 = CI(d,d-2) \cup CI(d,d-2)$ \\
2 & $Y_1 = CI(d,d-2) \cup CI(d,d-2)$ & $((d-2) + (d-2), 2d-3)$ & $Y_2 =  CI(d-3,d-2) \cup CI(d-3,d-2)$ \\
3 & $Y_2 = CI(d-3,d-2) \cup CI(d-3,d-2)$ & $((d-3) + (d-3),2d-6)$ & $Y_3 = CI(d-4,d-3) \cup CI(d-4,d-3)$ \\
4 &  $Y_3 = CI(d-4,d-3) \cup CI(d-4,d-3)$ & $((d-4) + (d-4), 2d-8)$ & $Y_4 = CI(d-5,d-4) \cup CI(d-5,d-4)$   \\
$ \vdots$ & $\vdots$ & $\vdots$ & $\vdots$ \\
$d-2$ & $Y_{d-3} = CI(2,3) \cup CI(2,3)$ & $(2+2,4)$ & $Y_{d-2} = CI(1,2) \cup CI(1,2)$.

\end{tabular}
}
\end{center}
It is certainly true, since we can view the parallel links separately, that the residuals are scheme-theoretic unions of complete intersections. Unfortunately after the first couple of links we do not know that these complete intersections are reduced or disjoint, although we believe this to be true. In any case, the last link results in a zero-dimensional scheme of degree 4, together with the promised redundancy of the minimal free resolutions, and this will allow us to conclude just as in Theorem \ref{d=4}.

\vspace{.1in}

\noindent \underline{Step (ii)}: Next we look at the $h$-vectors. The $h$-vector for the starting set $Z$ is $\Delta h_Z:
 (0, 1, 2, 3,4 , \dots  , 2d-3, 2d-2, 2d-2 , d ,1) $
and  we can compute the sequence of $h$-vectors of the residuals. The following lists the sequence of $h$-vectors of the residuals.
{\scriptsize
\[
\begin{array}{c|cccccccccccccccccccc}
 & \hbox{deg:} \\
 & 0 & 1 & 2 & 3 & 4  & \dots  & 2d-10 & 2d-9 & 2d-8 & 2d-7 & 2d-6 & 2d-5 & 2d-4 & 2d-3 & 2d-2& 2d-1 & 2d \\ \hline
Y_0 & 1 & 2 & 3  & 4 & 5 & \dots & 2d-9 & 2d-8 & 2d-7 & 2d-6 & 2d-5 & 2d-4 & 2d-3 & 2d-2 & 2d-2 & d & 1 \\
Y_1 & 1 & 2 & 3 & 4 & 5 & \dots & 2d-9  & 2d-8 & 2d-7 & 2d-6 & 2d-5 & 2d-4 & 2d-4 & d-2 \\
Y_2 & 1 & 2 & 3 & 4 & 5  & \dots & 2d-9 & 2d-8 & 2d-7 & 2d-6 & d-3 \\
Y_3 & 1 & 2 & 3 & 4 & 5 & \dots & 2d-9 & 2d-8 & d-4 \\
Y_4 & 1 & 2 & 3 & 4 & 5 & \dots & d-5 \\
&&& \vdots \\
Y_{d-3} & 1 & 2 & 3 & 4 & 2 \\
Y_{d-2} & 1 & 2 & 1
\end{array}
\]
}

\vspace{.1in}

\noindent \underline{Step (iii)}: Next we derive the minimal free resolutions, assuming the existence of the links.
The main observation is that in each step, even if all conceivable summands split in the mapping cones,  there remains a redundant term in the free resolution;  if fewer splittings occur, there could only be more redundancy. The desired contradiction at the end is based on the existence of this redundancy.

The first complete intersection in the first  table links $Z = Y_0$ to a residual curve $Y_1 = Y_{1,1} \cup Y_{1,2}$, where $Y_{1,1}$ and $Y_{1,2}$ are complete intersections as described in the top of the fourth column in the table in Step (i). A mapping cone starting with (\ref{Z-d})  (and splitting the summand corresponding to the minimal generator of degree $2d-2$ used in the link)  gives a minimal free resolution
\[
  0 \rightarrow
  \begin{array}{c}
S(-2d+2)^{1+\epsilon} \\
  \oplus \\
S(-2d+1)^{d-2} \\
   \end{array}
  \rightarrow
  \begin{array}{c}
S(-2d+4) \\
  \oplus \\
S(-2d+3)^{d-2} \\
  \oplus \\
S(-2d+2)^{1+\epsilon} \\
   \end{array}
  \rightarrow  I_{Y_1} \rightarrow 0.
\]
(It is less obvious that the generator of degree $2d$ in the complete intersection is a minimal generator of $I_Z$, so we note that $\epsilon$ may have grown by one here, but by slight abuse of notation we continue to write $1+\epsilon $ with $\epsilon \geq 0$. This is the only time we will have this ambiguity.)

The next link, using a complete intersection of type $((d-2) + (d-2),2d-3)$, uses two minimal generators and the residual, $Y_2$, has minimal free resolution
\[
  0 \rightarrow
  \begin{array}{c}
S(-2d+5)^{1+\epsilon} \\
  \oplus \\
S(-2d+4)^{d-3} \\
   \end{array}
  \rightarrow
  \begin{array}{c}
S(-2d+6)^{d-2} \\
  \oplus \\
S(-2d+5)^{1+\epsilon} \\
   \end{array}
  \rightarrow  I_{Y_2} \rightarrow 0.
\]
Notice the redundant $S(-2d+5)$ and the fact that the next link will involve two forms of degree $2d-6$, hence we cannot split off any redundant terms from this resolution in the  next mapping cone. One checks that this redundancy persists in the same way until the end.

\vspace{.1in}

\noindent \underline{Step (iv)}: Notice that in the desired sequence of links described in Step (i) the first two are a bit different from the rest, in that the third and subsequent links are by forms of the same degree (one reducible) whereas the first two are not of this form. Now we justify the existence and important properties of the first two links.

\begin{proposition} \label{link 1}
The first link exists. That is, $Y_0$ is linked by a complete intersection $(G_{0,1}G_{0,2}, F_0)$ of type $(2d, 2d-2)$ to $Y_1$, where $G_{0,1}$ is a general element of $[I_{Y_{0,1}}]_d$ and $G_{0,2}$ is a general element of $[I_{Y_{0,2}}]_d$; both are smooth. The residual, $Y_1 = Y_{1,1} \cup Y_{1,2}$ is reduced. In particular $Y_{1,1}$ and $Y_{1,2}$ have no points in common, and both are complete intersections of type $(d,d-2)$. The unique form $F_0$ of degree $2d-2$ in $I_{Y_0}$ is irreducible.
\end{proposition}

\begin{proof}

We have that $Z = Y_0 = Y_{0,1} \cup Y_{0,2}$ is the general hyperplane section of the curve $C = C_1 \cup C_2$ in $\mathbb P^3$, and each of the latter two curves is a smooth complete intersection curve in $\mathbb P^3$ of type $(d,d)$. Hence both $Y_{0,1} $ and $Y_{0,2} $ are reduced complete intersection sets of points with UPP. Furthermore, there are  pencils $\mathcal P_{0,1}$ and $\mathcal P_{0,2}$ of forms of degree $d$ defining $Y_{0,1}$ and $Y_{0,2} $, respectively, so in particular  the base locus of each of these pencils is the corresponding complete intersection set of points. Hence  a general element in $\mathcal P_{0,1}$ and a general element in $\mathcal P_{0,2}$ will each be smooth.

We know from the $h$-vector that there is a unique form $F_0$ of degree $2d-2$ containing $Y_0$. If this form were not reduced then removing a factor would yield a curve of lower degree containing all of the points, which is ruled out by the $h$-vector. Hence $F_0$ is reduced.

We now claim that $F_0$ is smooth at the points of $Y_0$. By the uniform position principle (see \cite{EH}, page 85), the points of $Y_0$ are indistinguishable, so the alternative is that $F_0$ must be singular at each of the $d^2$ points of $Y_{0,1}$ and each of the $d^2$ points of $Y_{0,2}$. Let $G_{0,1}$ be a general element of $\mathcal P_{0,1}$ and let $G_{0,2}$ be a general element of $\mathcal P_{0,2}$. Since $\mathcal P_{0,1}$ and $\mathcal P_{0,2}$ each define a finite set of points, the product $G_0 = G_{0,1} G_{0,2}$ does not have any component in common with $F_0$. But then the complete intersection of $F_0$ and $G_0$ has degree at least $2 \cdot 2d^2 = 4d^2$, while the product of degrees is only $(2d-2)(2d) < 4d^2$. This eliminates the possibility that $F_0$ is singular at each of the points, and we are done with the claim.

It follows from what we have said that the first link exists. Furthermore, by choosing a general element of each pencil, $G_0$ will meet $F_0$ transversally at every point of intersection. Consequently for the linked set we have

\begin{center}
{\it $Y_{1,1}$ and $Y_{1,2}$ are reduced and disjoint from each other. }
\end{center}

\noindent The fact that $Y_{1,1}$ is a complete intersection of type $(d,d-2)$ follows immediately from the fact that $Y_{0,1}$ is a complete intersection of type $(d,d)$ and it is linked to $Y_{1,1}$ by a complete intersection of type $(d,2d-2)$; similarly $Y_{1,2}$ is a complete intersection of type $(d,d-2)$.

\vspace{.1in}

\noindent \underline{Claim}: {\it The unique form $F_0$ of degree $2d-2$ in $I_{Y_0}$ is irreducible}.

\vspace{.1in}

Suppose $F_0 = H_1\cdots H_r$ are the irreducible components of $G$. The $d^2$ points of $Y_{0,1}$ have UPP (since $C_1$ is a smooth complete intersection curve) and the same is true for $Y_{0,2}$.   Let $W_1$ be the subset of $Y_{0,1}$ lying on $H_1$ and let $W_2$ be the subset of $Y_{0,2}$ lying on $H_1$. Let $P \in W_1$ and choose $Q \in Y_{0,1}$ such that $Q \notin W_1$. In \cite{EH},  Harris showed that the monodromy groups for $C_1$ and for $C_2$ are both the full symmetric group, so there is some loop $\gamma$ in a suitable open set of $(\mathbb P^3)^*$ so that moving the hyperplane along $\gamma$ interchanges $P$ and $Q$ but leaves all other points of $Y_{0,1}$ fixed. We can further arrange it so that $\gamma$ does not move any points of $Y_{0,2}$ (it is enough that $\gamma$ avoid planes that are tangent to $C_2$).
But $W_1 \cup W_2$ determines $G_1$, so by monodromy $(W_1 \backslash \{P\}) \cup \{Q \} \cup W_2$ determines $H_1$. (If it did not, $F_0$ would not be unique.) This is a contradiction since $Q$ does not lie on $G_1$. This proves the Claim.
\end{proof}

\begin{proposition} \label{link 2}
The second link exists. That is, $Y_1$ is linked by a complete intersection $(G_{1,1}G_{1,2},F_1)$ of type $(2d-4, 2d-3)$ to $Y_2$, where $G_{1,1} \in [I_{Y_{1,1}}]_{d-2}$ and $G_{1,2} \in [I_{Y_{1,2}}]_{d-2}$. The residual, $Y_2 = Y_{2,1} \cup Y_{2,2}$ is reduced, and in particular $Y_{2,1}$ and  $Y_{2,2}$ have no points in common. Note that $G_{1,1} G_{1,2}$ is the unique element of $I_{Y_1}$ of degree $2d-4$.
\end{proposition}

\begin{proof}
Now $Y_1 = Y_{1,1} \cup Y_{1,2}$ is  the residual in the first link; note again that without loss of generality,  $Y_{1,1}$ and $Y_{1,2}$ are reduced and disjoint. We have seen that $Y_{0,1}$ is linked in a complete intersection of type $(d,2d-2)$ to the complete intersection $Y_{1,1}$ of type $(d,d-2)$, and similarly  $Y_{0,2}$ is linked to $Y_{1,2}$ of the same type. But above we computed the minimal free resolution of $Y_1$ and saw that there is a unique curve of degree $2d-4$ containing $Y_1$. Hence this curve is the union of the unique curve $G_{1,1}$ of degree $d-2$ containing $Y_{1,1}$, and the unique curve $G_{1,2}$ of degree $d-2$ containing $Y_{1,2}$;  both $G_{1,1}$ and $G_{1,2}$ are reduced since $Y_1$ is.

From the resolution  we know not only that $I_{Y_1}$ has exactly one minimal generator of degree $2d-4$ (namely $G_{1,1} G_{1,2}$), but also  exactly $d-2$ minimal generators of degree $2d-3$ and at least one minimal generator of degree $2d-2$. For the second link, we want to show that there is a regular sequence of type $(2d-4,2d-3)$ in $I_{Y_1}$. The alternative is that all the minimal generators of $I_{Y_1}$ of degree $\leq 2d-3$ have a common factor. For the sake of contradiction, suppose this were the case. Let $H$ be the common factor. Notice that $1 \leq \deg(H) < d-2$ since $H$ would have to be a factor of $G_{1,1}$ (and also of  $G_{1,2}$).
By slight abuse of notation, we will use $H$ to denote both the form and the corresponding curve.

$G_{1,1}$ and $G_{1,2}$ are not necessarily irreducible, but they have to act the same way, by the Symmetry Principle. By assumption, all forms of degree $\leq 2d-3$ in $I_{Y_1}$ have $H$ as a factor so in particular the product $G_{1,1} G_{1,2}$ has $H$ as a  factor. If any polynomial divides both  $G_{1,1}$ and $G_{1,2}$ then in fact we can divide $G_{1,1} G_{1,2}$ by this polynomial and get a polynomial of degree $< 2d-4$ containing $Y_1$, contradicting the Hilbert function of $Y_1$. Thus $H$ must be reducible, with one (not necessarily irreducible) factor dividing $G_{1,1}$ and the  other dividing $G_{1,2}$. By symmetry these factors must have the same degree, so $\deg( H)$ must be even, say $\deg(H) = 2 \ell$. Then there are reduced forms $P_1, P_2, M_1, M_2$  such that

\begin{itemize}

\item $\deg(P_1) = \deg(P_2) = \ell$,

\item $H = P_1 P_2$,

\item $G_{1,1} = M_1 P_1$,

\item  $G_{1,2} = M_2 P_2$,

\item  $G_{1,1}, G_{1,2}$ have no common factor, and

\item all forms in $I_{Y_1}$ of degree $\leq 2d-3$ have $P_1 P_2$ as a factor.

\end{itemize}

\noindent Now, $Y_{1,1}$ is the reduced complete intersection of $M_1 P_1$ and a form of degree $d$, so it contains exactly $d \ell$ points of $H$. Similarly, $Y_{1,2}$ contains $d \ell$ points of $H$. Then all together, we conclude for the common factor $H$ that
\begin{equation} \label{kd pts}
\hbox{{\it $H$ contains $2 \ell d $ points of $Y_1$, an even number.}}
\end{equation}

Using methods introduced by Davis \cite{davis}, consider the exact sequence
\begin{equation} \label{ses}
0 \rightarrow S/(I_{Y_1} : H) (-2\ell) \stackrel{\times H}{\longrightarrow} S/I_{Y_1} \rightarrow S/(I_{Y_1},H) \rightarrow 0.
\end{equation}
The ideal $I_{Y_1} : H$ is the  saturated ideal of the set of points of $Y_1$ not on the curve defined by $H$. The saturation of the ideal $(I_{Y_1,}H)$ defines the set of points of $Y_1$ on $H$. Let us denote by $\bar A$ the points of $Y_1$ on $H$ and by $A'$ the points not on $H$, so $Y_1 = \bar A \cup A'$.

Notice that
\begin{equation} \label{describe A}
[(I_{Y_1},H)]_t = [(H)]_t
\end{equation}
for all $t \leq 2d-3$. Thus from (\ref{ses}) we also get  for the $h$-vectors
\[
\Delta h_{S/I_{A'}}(t-2\ell) = \Delta h_{S/I_{Y_1}}(t) - 2\ell
\]
for all $2 \ell \leq t \leq 2d-3$.

We know that the $h$-vector of $S/I_{Y_1}$ is

{\scriptsize
\begin{equation} \label{hvtrY}
\begin{array}{c|cccccccccccccccccccc}
\hbox{deg:} & 0 & 1 & 2 & 3 & 4  & \dots  & 2d-10 & 2d-9 & 2d-8 & 2d-7 & 2d-6 & 2d-5 & 2d-4 & 2d-3 & 2d-2  \\ \hline
& 1 & 2 & 3 & 4 & 5 & \dots & 2d-9  & 2d-8 & 2d-7 & 2d-6 & 2d-5 & 2d-4 & 2d-4 & d-2 & 0\\
\end{array}
\end{equation}  }
 In particular, it is zero in degree $2d-2$.
From (\ref{ses}) we thus get that $\Delta h_{S/I_{A'}}(2d-2-2\ell) = 0$.
These facts imply that $h^1(\mathcal I_{Y_1}(2d-3)) = 0$ and $h^1(\mathcal I_{A'}(2d-3-2\ell)) = 0$.
Now from the exact sequence
\[
0 \rightarrow (I_{Y_1} : H)(-2\ell) \stackrel{\times H}{\longrightarrow} I_{Y_1} \rightarrow (I_{Y_1},H) \rightarrow 0,
\]
sheafifying and taking cohomology in degree $2d-3$, we get that $(I_{Y_1},H)$ is saturated in degrees $\geq 2d-3$, i.e. it agrees with $I_{\bar A}$ in  degrees $\geq 2d-3$. By (\ref{describe A}), we see that in degree $2d-3$ we have $(I_{Y_1},H) = (H)$, which we now know agrees with $I_{\bar A}$ in that degree. Then we can complete (\ref{describe A}) as follows:
\[
[(I_{Y_1},H)]_t = [(H)]_t = [(I_{Y_1},H)^{sat}]_t =  [I_{\bar A}]_t
\]
for all $t \leq 2d-3$.

It follows that $S/I_{\bar A}$ has $h$-vector
{\scriptsize
\[
\begin{array}{c|cccccccccccccccccccc}
\hbox{deg:} & 0 & 1 & 2 & 3 &  \dots & 2\ell-2 &  2\ell-1 & 2\ell & 2\ell+1 & \dots   & 2d-6 & 2d-5 & 2d-4 & 2d-3 & 2d-2   \\ \hline
& 1 & 2 & 3 & 4 & \dots & 2\ell-1 &  2\ell &2\ell &2\ell & \dots & 2\ell & 2\ell & 2\ell &2\ell & 0 \\
\end{array}
\]
 }

\noindent In particular, $Y_1$ has exactly
\[
2\ell ( 2d-2\ell-1) + \binom{2\ell}{2} = 2\ell \left [  (2d-2\ell-1) + \frac{2\ell-1}{2} \right ]
\]
points on $H$. But in (\ref{kd pts}) we computed that there are $2\ell d$ points of $Y_1$ on $H$, so
\[
d = (2d-2\ell-1) + \frac{2\ell-1}{2} .
\]
But $2\ell$ is even, so the righthand side is not even an integer, and we have a contradiction.
Hence there is no common factor, and the second link also exists: $Y_1$ is linked to $Y_2$ by a complete intersection of type $(2d-4, 2d-3)$.
It is clear from the uniqueness of the  form of initial degree $2d-4$ that both forms participating in this link are minimal generators of $I_{Y_1}$.
\end{proof}

\vspace{.1in}

The first two links were special (numerically), but now all the rest of the links follow the same pattern.
For the $i$-th links, $i \geq 3$, we want to link $Y_{i-1} = Y_{i-1,1} \cup Y_{i-1,2}$ $(i \geq 3)$ to $Y_i = Y_{i,1} \cup Y_{i,2}$ using a complete intersection of type $((d-i) + (d-i), 2d-2i)$. The ideal of $Y_{i-1}$ has $d-i+1$ minimal generators in degree $2d-2i$ (the initial degree), so certainly if the regular sequence exists, the link is done with two minimal generators (and so the residual has one fewer minimal generator).  One of these minimal generators is the union of a curve $N_1$ of degree $d-i$ containing $Y_{i-1,1}$ and a curve $N_2$ of degree $d-i$ containing $Y_{i-1,2}$.

Suppose first that all the desired links from Step (i) exist. We end with a scheme $Y_{d-2}$ of degree 4, whose minimal free resolution has a redundant term. This can only be a copy of $S(-3)$ in both free modules, i.e. the minimal free resolution of $Y_{d-2}$ is
\[
0 \rightarrow
\begin{array}{c}
S(-3) \\
\oplus \\
S(-4)
\end{array}
\rightarrow
\begin{array}{c}
S(-2)^2 \\
\oplus \\
S(-3)
\end{array}
\rightarrow I_{Y_{d-2}} \rightarrow 0.
\]
Even if all the links exist, we do not claim that the residuals continue to be reduced, even though we verified this in the first two links (Propositions \ref{link 1} and \ref{link 2}). In particular, $Y_{d-2}$ may be non-reduced. However, it must have degree 4, and it must contain a scheme $Y_{d-2,1}$ of degree 2 and a scheme $Y_{d-2,2}$ of degree 2, obtained via the parallel links.

\vspace{.1in}

\noindent { \underline{Claim}:} {\it $Y_{d-2}$ must have a subscheme of degree 3 lying on a line, which gives a contradiction}.

Indeed, we have seen that $I_{Y_{d-2}}$ has two minimal generators of degree 2 and one of degree 3, and that it has degree 4. If there were a regular sequence of two forms of degree 2, $Y_{d-2}$ would be a complete intersection, contradicting what we know to be the minimal free resolution. Thus the forms of degree 2 have a degree 1 common divisor. By Davis's theorem $Y_{d-2}$ has a subscheme of degree 3 lying on this line. But then either $Y_{d-2,1}$ or $Y_{d-2,2}$, but not both, must lie on this line. This violates the Symmetry Principle, giving the desired contradiction. This would conclude not only the proof of the Claim, but in fact the proof that $R/I$ has the WLP (always assuming Conjecture \ref{force exactly one}).

\vspace{.1in}

The last issue is to deal with the existence of the remaining links. For any $i$, with $3 \leq i \leq d-2$, if the $i$-th link does not exist then certainly there is a common factor in the initial degree of the ideal of $Y_{i-1}$.

\begin{conjecture} \label{get contra}
For any $i$, $3 \leq i \leq d-2$, if the $i$-th link does not exist, i.e., there is a common factor in the initial degree of the ideal of $Y_{i-1}$,   then this common factor defines  a curve that contains
enough points of $Y_{i-1}$ so that one obtains a contradiction using the Symmetry Principle in the same way as was done in the proof of Proposition \ref{link 2}.
\end{conjecture}


\section{Some additional results and consequences}

\subsection{Jacobian ideals} \label{jacobian ideals}

We now give a consequence of our results to Jacobian ideals. We recall that if $X: f=0$ is a smooth hypersurface of degree $d+1$ in $\mathbb P^n$ defined by a homogeneous polynomial $f$ then the Jacobian ideal $J(X)$ (or $J(f)$) is the homogeneous ideal generated by the  partial derivatives $\frac{\partial f}{\partial x_i}$, $0 \leq i \leq n$, and it is a complete intersection generated by forms of degree $d$ since $X$ is smooth.

In \cite{ilardi} G. Ilardi posed the following question:

\begin{quotation}
{\it Does the Jacobian ideal of a smooth hypersurface have the Weak Lefschetz Property?}
\end{quotation}

Ilardi proves the following partial answer. We will slightly change her notation to agree with ours. For a graded algebra $A$, {\it having WLP in degree $t$} will mean that the multiplication map $\times L : [A]_{t} \rightarrow [A]_{t+1}$ has maximal rank.

\begin{proposition}[Ilardi \cite{ilardi}]
Let $X : f = 0$ be a hypersurface in $\mathbb P^n$ of degree $d+1 > 2$, such that its singular locus $X_s$ has dimension at most $n -3$. Then the ideal $J(X)$ has the WLP in degree $d -1$.
\end{proposition}

In \cite{AR}, A. Alzati and R. Re proved this same injectivity for any complete intersection generated by forms of degree $d$, not only for Jacobian ideals.

We  give some  answers to Ilardi's question arising from our work. First we note that our results show that in some cases $R/J(X)$ has the full WLP rather than WLP in a certain degree.

\begin{corollary}
Let $F$ be a smooth hypersurface in $\mathbb P^3$ of degree 3, 4, 5 or 6. Then the Jacobian ideal $J = \langle \frac{\partial F}{\partial w}, \frac{\partial F}{\partial x}, \frac{\partial F}{\partial y}, \frac{\partial F}{\partial z}\rangle$ has the WLP.
\end{corollary}

Using Theorem \ref{d1=d2=d3=d4} we also get an improvement of Ilardi's result in the case of four variables:

\begin{corollary} \label{jacobian result}
Let $X : f = 0$ be a smooth hypersurface in $\mathbb P^3$ of degree $d+1 > 2$. Then the ideal $J(X)$ has the WLP in all degrees $\leq
  \lfloor\frac{3d+1}{2}\rfloor -2$, i.e. $\times \ell : [R/J(X)]_t \rightarrow [R/J(X)]_{t+1}$ is injective for all $t \leq \lfloor \frac{3d+1}{2} \rfloor - 2$.
\end{corollary}

We remark that WLP is equivalent to injectivity for all $t \leq 2d-3$, so Corollary \ref{jacobian result} covers approximately half the range left open by the Ilardi and Alzati-Re result.


\subsection{A result for non-equigenerated complete intersections } \label{not equigenerated subsec}

Even if we prove Conjectures \ref{force exactly one} and \ref{get contra}, it would remain to prove that {\it all} codimension four  complete intersections  (with generators of possibly different degree) have the WLP.
In the last proposition of this paper we deal with complete
intersections of arbitrary degree $d_1,d_2,d_3,d_4$ and from the
stability of the associated syzygy bundle we will deduce the
injectivity in a range unfortunately not as good as the one that we
get when $d_1 = \dots = d_4$ (see Theorem \ref{d1=d2=d3=d4}). However,  it does not assume that the degrees are equal, and  it introduces a different approach.

For the sake of completeness we recall the following result on vector bundles that will be crucial in the proof of Proposition \ref{d1>=d2>=d3>=d4}.

\begin{proposition}[\cite{EHV} Theorem 3.4 (due to Schneider), and \cite{EHV} Theorem 6.1] \label{EHV results}

Let $\mathcal E$ be a normalized (i.e. $-2\le c_1(\mathcal E) \le 0$) rank 3 stable vector bundle on $\mathbb P^3$. The restriction of $\mathcal E$ to a general plane is stable unless one of the following holds:

\begin{enumerate}
\item[(i)] $c_1(\mathcal E) = -2$ and $\mathcal E = T_{\PP^3}(-2) $, where $T_{\PP^3}$ is the tangent bundle on $\PP^3$;

\item[(ii)] $ c_1(\mathcal E) = -1$ and $\mathcal E = \Omega ^1(1)$, where $\Omega = \Omega_{\PP^3}^1$ is the sheaf of K\"ahler differentials;

\item [(iii)] $c_1 (\mathcal E) = 0$ and $c_2(\mathcal E) \le 3$;

\item[(iv)] $c_1(\mathcal E) = 0$ and $\mathcal E = S^2(\mathcal N)$ where $\mathcal N$ is the null correlation bundle and $S^2(\mathcal N)$ is the second symmetric power;

\item[(v)] $c_1(\mathcal E) = 0$ and $\mathcal E$ fits in the exact sequence
\[
0 \rightarrow \Omega (1) \rightarrow \mathcal E  \rightarrow O_{H_0}(-c_2(\mathcal E)+1) \rightarrow 0
\]
for some plane $H_0$ in $\PP^3$.

\end{enumerate}

\end{proposition}

\begin{proposition} \label{d1>=d2>=d3>=d4} Let $A = R/I = R/\langle
  F_1,F_2,F_3,F_4 \rangle$ where $I$ is a complete intersection and
  $\deg F_i = d_i$. Set $d_1+d_2+d_3+d_4=3\lambda +r$, $0\le
  r\le2$. Let $L$ be a general linear form. Then the multiplication
  maps $\times L \colon [A]_{t-1}\rightarrow[A]_t$ are injective
  for $t< \lambda$.
\end{proposition}

\begin{proof}
  We will assume that $d_1\le d_2\le d_3\le d_4$.  We distinguish two
  cases.
  \begin{enumerate}
  \item If $\frac{d_1+d_2+d_3+d_4}{3}\le d_4$, then $d_1 + d_2+d_3+d_4
    = 3 \lambda + r \leq 3d_4$, so $\lambda \leq d_4$. If $t <
    \lambda$ then $[A]_{t-1}$ and $[A]_t$ coincide with the corresponding components of the coordinate
    ring of a complete intersection of positive Krull dimension, so
    the result is obvious.
  \item Assume that $\frac{d_1+d_2+d_3+d_4}{3}> d_4$. Consider the
    syzygy bundle
    \[
    {\mathcal E}:= \ker \left (\bigoplus _{i=1}^4{\mathcal O}_{\PP^3}(-d_i)
    \stackrel{(F_1,F_2,F_3,F_4)}{\longrightarrow} {\mathcal
      O}_{\PP^3} \right )
    \]
    associated to $(F_1,F_2,F_3,F_4) $.  ${\mathcal E}$ is a rank 3
    vector bundle on $\PP^3$ with $c_1( {\mathcal E})=-(
    d_1+d_2+d_3+d_4)$. Note that $H^1_*(\mathcal E) = \bigoplus_{t \in \mathbb Z} H^1(\mathcal E(t)) \cong A$ (cf.\ \cite{BK} Proposition 2.1, although this was already used implicitly in \cite{HMNW} Theorem 2.3).  Let us check that ${\mathcal E}$ is $\mu
    $-stable. To this end, we consider the exact sequence
    \[
    0 \longrightarrow {\mathcal O}_{\PP^3} \longrightarrow \bigoplus
    _{i=1}^4{\mathcal O}_{\PP^3}(d_i) \longrightarrow {\mathcal
      E}^\ast\longrightarrow 0.
    \]
    By hypothesis we have
    \[
    \mu({\mathcal E}^\ast):=\frac{c_1({\mathcal E}^\ast)}{rk({\mathcal
        E}^\ast)}=\frac{\sum_{i=1}^4d_i}{3}> \max \{d_i \}=d_4.
    \]
    So we can apply \cite{bs} Corollary 2.7, and conclude that
    ${\mathcal E}^\ast$ is $\mu $-stable. Since, $\mu$-stability is
    preserved under dualizing we also have that ${\mathcal E}$ is $\mu
    $-stable.

    Now we want to consider the restriction to a general plane $H$. We claim that by Proposition \ref{EHV results}, the restriction
    ${\mathcal E}_{|H}$ of ${\mathcal E}$ to  $H\subset
    \PP^3$ is also $\mu$-stable.
This is because  our rank 3 vector  bundle  is not one of the few exceptions listed in that result. Indeed,  recall that  $H^1_*(\mathcal E)$ is isomorphic to our artinian algebra $R/I$. The non-zero summands of $R/I$ go from the homogeneous part of degree zero until the homogeneous part of degree $d_1+d_2+d_3+d_4-4$ (assuming $d_i\ge 2$; if one is smaller than 2 we  are dealing with a complete intersection in 3 variables and the result is known). Moreover we know exactly the dimension of the homogeneous part of degree $i$, for $0\le i\le d_1+..+d_4-4$. They are $(1,4,h_2,\dots,h_2,4,1)$.

We claim that our vector bundle $\mathcal E$ is not one of the exceptions listed in Proposition \ref{EHV results}. Exception (i) has $H^1_*(\mathcal E)=0$, while (ii) has $H^1_*(\mathcal E)$ concentrated in only {\it one} degree. So in neither case do we have $H^1_*(\mathcal E) \cong R/I$ (up to twist).

For the remaining three exceptions, (v) corresponds  again to a bundle $\mathcal E$ (or $\mathcal E^*$) with $H^1_*(\mathcal E)$ concentrated in only {\it one} degree and all the others are 0,  which is not our case  -- indeed, only $H^1 \mathcal E(-1)\ne 0$.

Exception  (iv) corresponds to the second symmetric power of the null correlation bundle $\mathcal N$. We claim that again is not our case because the cohomology satisfies  $\dim H^1(S^2(\mathcal N)(-2))=1$, $\dim H^1(S^2(\mathcal N)(-1))= 4$ and $\dim H^1(S^2(\mathcal N))=5$, and this cannot be the start of the Hilbert function of a complete intersection. Indeed, the null correlation bundle is a rank 2 vector bundle $\mathcal N$ on $\mathbb P^3$ with $c_1(\mathcal N) = 0$. Therefore we have an exact sequence
\[
0 \rightarrow \bigwedge^2 \mathcal N \rightarrow \mathcal N \otimes \mathcal N \rightarrow S^2 (\mathcal N) \rightarrow 0,
\]
and $\bigwedge^2 \mathcal N \cong \mathcal O_{\mathbb P^3}(c_1 (\mathcal N)) \cong \mathcal O_{\mathbb P^3} $. We deduce that $H^1 ((\mathcal N \otimes \mathcal N) (t)) \cong H^1 (S^2 (\mathcal N)(t))$ for all $t \in \mathbb Z$, from which the result follows from a calculation.

Finally, we consider exception (iii).  This corresponds to $c_1=0$ and $c_2=3$ and again we claim this is not our case. Indeed, assume without loss of generality  that $2 \le d_1 \le d_2 \le d_3 \le d_4$. We know that $c_1 (\mathcal E) = -(d_1+d_2+d_3+d_4)$ and 
\begin{equation} \label{c2}
c_2(\mathcal E) = d_1d_2+d_1d_3+...+d_3d_4.
\end{equation}

We have $c_1(\mathcal E_{norm}) = 0$ if and only if $c_1(\mathcal E) \equiv 0$ (mod 3). Thus we can write
\[
-(d_1 + d_2 + d_3 + d_4) = -3p
\]
and in particular  $-(d_1+d_2+d_3+d_4) \equiv 0 \ (~\hbox{mod } 3)$.
At the beginning of this proof we divided into two cases, and the  current case is $\frac{d_1+d_2+d_3+d_4}{3}> d_4$. Thus $\frac{d_1+d_2+d_3+d_4}{3} \geq d_4+1$, i.e. 
\begin{equation} \label{other one}
d_1 + d_2 + d_3 \ge 2d_4 + 3.
\end{equation}

Since $p = (d_1+d_2+d_3+d_4)/3 $ we have
$c_1(\mathcal E_{norm}) = c_1 (\mathcal E(p)) = 0$   and
\[
c_2(\mathcal E_{norm}) = c_2 (\mathcal E(p)) = c_2(\mathcal E) - 3p^2 = c_2(\mathcal E) - \frac{1}{3} \left (d_1^2 + d_2^2 + d_3^2 + d_4^2 + 2 \cdot \sum_{1 \leq i<j \leq 4} d_i d_j \right ).
\]
We will use  (\ref{c2}) and (\ref{other one}), as well as the following inequalities:
\[
d_1 d_2 \geq d_1^2, \ d_2 d_3 \geq d_2^2, \ d_4^2 \geq d_3^2.
\]
Thus
\[
\begin{array}{lcl}
c_2(\mathcal E_{norm}) & = & \displaystyle \frac{1}{3} \left [ \left (\sum_{1 \leq i < j \leq 4} d_i d_j \right ) - d_1^2 - d_2^2 - d_3^2 - d_4^2 \right ] \\
& = & \displaystyle  \frac{1}{3} \left [ d_4 (d_1 + d_2 + d_3) + \left ( \sum_{1 \leq i < j \leq 3} d_i d_j \right )  - d_1^2 - d_2^2 - d_3^2 - d_4^2 \right ]  \\ \vspace{.05in}
& \geq &  \frac{1}{3} \left [ d_4 (2d_4+3)  + d_1d_2 + d_1 d_3 + d_2 d_3 - d_1^2 - d_2^2 -d_3^2 - d_4^2    \right ] \\ \vspace{.05in}
& \geq &  \frac{1}{3} (3d_4 + d_1 d_3) \\
& \geq &  \frac{1}{3} (3 \cdot 2 + 2 \cdot 2) > 3.
\end{array}
\]
Thus our vector bundle $\mathcal E$ does not fall into any of the exceptions listed in Proposition \ref{EHV results}, and $\mathcal E_{|H}$ is stable for a general plane $H$.
In particular, we have $H^0(H, {\mathcal E}_{|H}(t ))=0$ for all $t<\lambda$. Looking at the long exact sequence in cohomology of the exact sequence
\[
0 \rightarrow \mathcal E(t-1) \stackrel{\times L}{\longrightarrow} \mathcal E(t) \rightarrow \mathcal E_{|H}(t) \rightarrow 0
\]
we see that if $L$ is a linear form defining $H$ then $\times  L : [A]_{t-1} \rightarrow [A]_t$ is injective for all $t<\lambda$.
\end{enumerate}
\end{proof}

\section{Final comments and questions}

\begin{enumerate}

\item The most obvious open question is whether the WLP holds for arbitrary complete intersections in arbitrarily many variables. This is the Holy Grail of this line of investigation. So far it is known in two or three variables (\cite{HMNW}), the complete intersection of at most four quadrics (\cite{MN-quadrics}) and now we have results for complete intersection of forms of the same degree $d\le 5$ in four variables. It might be profitable to apply the methods of this paper to complete intersections generated by six forms of the same degree $d$ in six variables. That means studying the union of two complete intersection surfaces in $\mathbb P^5$, whose general hyperplane section is the union of two complete intersection curves in $\mathbb P^4$.  Of course such a union in $\mathbb P^4$ is not ACM, as was the case for the hyperplane sections in this paper.

\vspace{.1in}

\item Once we show that all complete intersections have the WLP, it remains to show that they all have the SLP. This is open even in codimension 3, but it is known in codimension 2 \cite{HMNW}. It is worth noting that results of Dimca, Gondim and Ilardi \cite{DGI} give SLP for the Jacobian ring of a smooth cubic surface in $\mathbb P^3$ and of a smooth quartic curve in $\mathbb P^2$.
Also,  a result of Bricalli, Favale and Pirola \cite{BFP} gives SLP for the complete intersection of five forms of degree 2 in $k[x_0,\dots,x_4]$; in particular, the Jacobian ring of a smooth cubic 3-fold in $\mathbb P^4$ has the SLP. 
\end{enumerate}


\end{document}